\documentclass[12pt]{amsart}
\usepackage[T2A]{fontenc}
\usepackage[english]{babel}
\usepackage{amsaddr}
\usepackage{amsmath}
\usepackage{amssymb}
\usepackage{amsfonts}
\usepackage{epsfig}
\usepackage{srcltx}
\usepackage{subfigure}
\usepackage{cite}
\usepackage{float}
\usepackage{mathtools}
\usepackage[a4paper, mag=1000, includefoot, left=2cm, right=2cm, top=2cm, bottom=2cm, headsep=1cm, footskip=1cm]{geometry}

\newtheorem{Th}{Theorem}
\newtheorem{Lem}{Lemma}

\begin{document}
\thispagestyle{empty}

\title[Long-term behaviour of asymptotically autonomous Hamiltonian systems]{Long-term behaviour of asymptotically autonomous Hamiltonian systems with multiplicative noise}
\author[O.A. Sultanov]{Oskar A. Sultanov}

\address{
Chebyshev Laboratory, St. Petersburg State University, 14th Line V.O., 29, Saint Petersburg 199178 Russia;\\
Institute of Mathematics, Ufa Federal Research Centre, Russian Academy of Sciences, Chernyshevsky street, 112, Ufa 450008 Russia.}
\email{oasultanov@gmail.com}


\maketitle
{\small

{\small
\begin{quote}
\noindent{\bf Abstract.} 
The influence of multiplicative stochastic perturbations on the class of asymptotically Hamiltonian systems on the plane is investigated. It is assumed that disturbances do not preserve the equilibrium of the corresponding limiting system and their intensity decays in time with power-law asymptotics. The paper discusses the long-term asymptotic behaviour of solutions and its dependence on the structure and parameters of perturbations. In particular, it is shown that perturbed trajectories can tend to the equilibrium of the limiting system or new stochastically stable states can arise. The proposed analysis is based on a combination of the averaging method and the construction of stochastic Lyapunov functions. The results obtained are applied to the problem of capture into parametric autoresonance in the presence of noise.
\medskip



\end{quote}
}

\section{Introduction}

In this paper, the influence of stochastic perturbations on a class of autonomous systems in the plane is investigated. It is assumed that the unperturbed system is Hamiltonian with a neutrally stable equilibrium, which is not preserved in the perturbed system. It is known that in this case even small stochastic perturbations can lead to the exit of trajectories from any bounded domain~\cite{FW98}. We assume that the intensity of perturbations decays in time so that the perturbed system is asymptotically autonomous, and we study the long-term behaviour of perturbed trajectories near the equilibrium of the limiting system. 

Dynamical systems with perturbations decaying in time have been studied in many papers. See, for example~\cite{LDP74}, where a class of perturbations is described that do not violate the qualitative behaviour of solutions of oscillatory Hamiltonian systems. In the general case, the long-term behaviour of solutions of asymptotically autonomous systems and the corresponding limiting systems can differ significantly~\cite{HRT94,OS20arxiv}. Bifurcations in deterministic asymptotically autonomous systems were discussed in~\cite{LRS02,KS05,MR08,OS21DCDS,OSIJBC21}. 

The effect of white noise disturbances on qualitative properties of solutions was studied in~\cite{RKh64,BG02,AMR08,BKGT08,DNR11,BHW12,TW15} for systems of stochastic differential equations with time-independent coefficients. Bifurcation phenomena in such systems, associated with changes in the profile of stationary probability densities, in the Lyapunov spectrum function or in the dichotomy spectrum, were investigated in~\cite{NSN90,CF98,BRS09,BRR15,CDLR17,DELR18}. The influence of time-dependent stochastic perturbations on the long-term behaviour of solutions in scalar systems was discussed in~\cite{AGR09,ACR11,KT13}. Decaying stochastic perturbations that preserve the equilibrium of the limiting Hamiltonian system were considered in~\cite{OS21arxiv}, where bifurcations associated with a change in the stability of the equilibrium were discussed. To the best of the author's knowledge, bifurcations in asymptotically autonomous systems with stochastic perturbations that do not preserve the equilibrium of the limiting system have not been studied in detail. In this paper, we describe possible stable asymptotic regimes in the perturbed system and their dependence on the structure and parameters of disturbances. 

The paper is organized as follows. In Section~\ref{PS}, the statement of the mathematical problem is given and the
class of damped perturbations is described. The main results are presented in Section~\ref{MainRes}. The proofs are
contained in the subsequent sections. In particular, in Section~\ref{Sec3} we construct changes of variables that simplify the system in the first asymptotic terms at infinity in time. These transformations are based on the transition to energy-angle variables related to the parameters of the general oscillating solution of the limiting Hamiltonian system, and a specific near-identity transformation of the energy variable. Under some reasonable assumptions on the the structure of the corresponding simplified equations, possible stable asymptotic regimes in the perturbed system are described. The justification of these results based on the construction of stochastic Lyapunov functions is contained in sections \ref{Sec4}, \ref{Sec5}, \ref{Sec6} and \ref{Sec7}. Section \ref{SecEx} discusses examples of nonlinear systems with damped stochastic perturbations. The application of the proposed theory to the autoresonance phenomenon is contained in Section~\ref{Appl}. The paper concludes with a brief discussion of the results obtained.

\section{Problem statement}\label{PS}
Consider the system of It\^{o} stochastic differential equations
\begin{equation}\label{FulSys}
	d {\bf z} = {\bf a}({\bf z},t) dt + {\bf A}({\bf z},t)\,d{\bf w}(t), \quad t>s>0, \quad {\bf z}(s)={\bf z}_0\in\mathbb R^2,
\end{equation}
where ${\bf z}=(x,y)^T$, ${\bf a}({\bf z},t)=(a_1(x,y,t),a_2(x,y,t))^T$ is a vector function, ${\bf A}({\bf z},t)=\{\alpha_{i,j}(x,y,t)\}_{2\times 2}$ is a $2\times2$ matrix, and ${\bf w}(t)=(w_1(t),w_2(t))^T$ is a two dimensional Wiener process on a probability space $(\Omega,\mathcal F,\mathbb P)$. The functions $a_i(x,y,t)$ and $\alpha_{i,j}(x,y,t)$, defined for all $(x,y,t)\in\mathbb R^2\times \mathbb R_+$, are infinitely differentiable and do not depend on $\omega\in\Omega$. It is assumed that
\begin{gather}\label{zero}
{\bf a}(0,t)\equiv 0, \quad {\bf A}(0,t)\not\equiv 0,
\end{gather}
and there exists $M>0$ such that ${\bf a}({\bf z},t)$ and ${\bf A}({\bf z},t)$ satisfy the Lipschitz and growth conditions:
\begin{gather}\label{lip}
\begin{split}
&|{\bf a}({\bf z}_1,t)-{\bf a}({\bf z}_2,t)|\leq M |{\bf z}_1-{\bf z}_2|, \\ 
&\|{\bf A}({\bf z}_1,t)-{\bf A}({\bf z}_2,t)\|\leq M |{\bf z}_1-{\bf z}_2|, \quad 
\|{\bf A}({\bf z},t)\|\leq M (1+ |{\bf z}|)
\end{split}
\end{gather}
for all ${\bf z}, {\bf z}_1, {\bf z}_2\in \mathbb R^2$ and $t\geq s$. Here, $|{\bf z}|=\sqrt{x^2+y^2}$ and $\|\cdot\|$ is the operator norm coordinated with the norm $|\cdot|$ of $\mathbb R^2$. Note that these restrictions on the coefficients of system ensure the existence and uniqueness of a continuous with probability one solution ${\bf z}(t)=(x(t),y(t))^T$ for all $t\geq s$ and for any initial point ${\bf z}_0\in\mathbb R^2$ (see, for example,~\cite[\S 5.2]{BOks98}).

In addition, it is assumed that system \eqref{FulSys} is asymptotically autonomous: for any domain $\mathcal D\subset\mathbb R^2$  
\begin{equation*}
	\lim_{t\to\infty} {\bf a}({\bf z},t)={\bf a}_0({\bf z}), \quad \lim_{t\to\infty} {\bf A}({\bf z},t)={\bf 0} 
\end{equation*}
for all ${\bf z}\in \mathcal D$. The limiting system
\begin{equation}\label{LimSys}
	\frac{d{\bf z}}{dt}={\bf a}_0({\bf z}), \quad {\bf a}_0({\bf z})\equiv \begin{pmatrix} \partial_y H_0(x,y) \\ -\partial_x H_0(x,y) \end{pmatrix}
\end{equation}
is assumed to be Hamiltonian with a fixed point at the origin $(0,0)$. Moreover, we assume that 
\begin{equation}\label{H0as}
	H_0(x,y)=\frac{|{\bf z}|^2}{2}+\mathcal O(|{\bf z}|^3), \quad |{\bf z}|\to 0,
\end{equation}
and there exist $E_0>0$ and $r>0$ such that for all $E\in (0,E_0]$ the level lines $\{(x,y)\in\mathbb R^2: H_0(x,y)=E\}$, lying in the ball $\mathcal B_r=\{(x,y)\in\mathbb R^2: |{\bf z}|\leq r\}$, are closed curves on the phase plane $(x,y)$. Note that each of these curves corresponds to a periodic solution ${\bf z}_\ast(t)\equiv (x_\ast(t,E), y_\ast(t,E))^T$ of system \eqref{LimSys} with the period $T(E)=2\pi/\nu(E)$, where $\nu(E)\neq 0$ for all $E\in (0,E_0]$ and $\nu(E)=1+\mathcal O(E)$ as $E\to 0$. In this case, the value $E=0$ corresponds to the equilibrium ${\bf z}(t)\equiv 0$. We also assume
that $\mathcal B_r$ does not contain any fixed points of system \eqref{LimSys}, except for the origin.

Damped perturbations of system \eqref{LimSys} are described by functions with power asymptotic expansions:
\begin{equation}\label{HFBas}
	\begin{split}
		  {\bf a}({\bf z},t)={\bf a}_0({\bf z})+\sum_{k=1}^\infty t^{-\frac{k}{q}} {\bf a}_k({\bf z}), \quad
		  {\bf A}({\bf z},t)=\sum_{k=1}^\infty t^{-\frac{k}{q}} {\bf A}_{k}({\bf z}),  \quad t\to\infty
	\end{split}
\end{equation}
for all ${\bf z}\in \mathcal B_r$ with some $q\in\mathbb Z_+$ and time-independent coefficients ${\bf a}_k({\bf z})$ and ${\bf A}_k({\bf z})=\{\alpha^k_{i,j}(x,y)\}_{2\times 2}$ such that 
\begin{gather}\label{Azero}
\exists\,p\in\mathbb Z_+: \quad  {\bf A}_k(0)=0, \quad k<p, \quad {\bf A}_{p}(0)\neq 0.
\end{gather}  

Note that decaying perturbations with power-law asymptotics appear in many problems associated with non-linear and non-autonomous systems~\cite{IKNF06,BG08,KF13,LK14,CYZZ18,OS21,CH21}.
The simplest example is given by the following linear system:
\begin{equation}\label{ex0}
\begin{split}
	 dx=ydt,\quad
	 dy=(-x+t^{-1}a y)\,dt+ t^{-\frac pq} c \,dw_2(t), \quad t>1,
\end{split}
\end{equation}
with $a,c={\hbox{\rm const}}$. This system is of the form \eqref{FulSys} with 
\begin{gather*}
{\bf a}({\bf z},t)\equiv {\bf a}_0({\bf z})+t^{-1}\begin{pmatrix}0\\ a y\end{pmatrix}, \quad
{\bf A}({\bf z},t)\equiv t^{-\frac{p}{q}} {\bf A}_p({\bf z}), \quad {\bf A}_p({\bf z})\equiv \begin{pmatrix} 0&0 \\0 &c \end{pmatrix}, \quad H_0(x,y)\equiv \frac{|{\bf z}|^2}{2}. 
\end{gather*}
It can easily be checked that the limiting system with $a=c=0$ has $2\pi$-periodic general solution $x_\ast(t+\phi;E)=\sqrt{2E}\cos (t+\phi)$, $y_\ast(t+\phi;E)=-\sqrt{2E}\sin (t+\phi)$, where $E$ and $\phi$ are arbitrary constants. In the absence of the stochastic term ($a\neq 0$ and $c=0$), the asymptotics of solutions can be constructed using the WKB approximations (see, for example,~\cite{WW66}): 
\begin{gather*}
	x(t)=t^{\frac{a}{2}}\left(x_\ast(t+\phi;E)+\mathcal O(t^{-1})\right),\quad 
	y(t)=t^{\frac{a}{2}}\left(y_\ast(t+\phi;E)+\mathcal O(t^{-1})\right),\quad 
	t\to\infty.
\end{gather*}
We see that in this case the stability of the equilibrium $(0,0)$ of the limiting system depends on the sign of the parameter $a$ (see Fig.~\ref{Fig12}, a). Numerical analysis of system \eqref{ex0} with $a\neq 0$ and $c\neq 0$ shows that the long-term behaviour of the solutions depends both on the values of the parameters $a$ and $c$ and on the degree of damping of the stochastic perturbation $p/q$ (see Fig.~\ref{Fig12}, b, c). 
\begin{figure}
\centering
\subfigure[$c=0$]{
\includegraphics[width=0.3\linewidth]{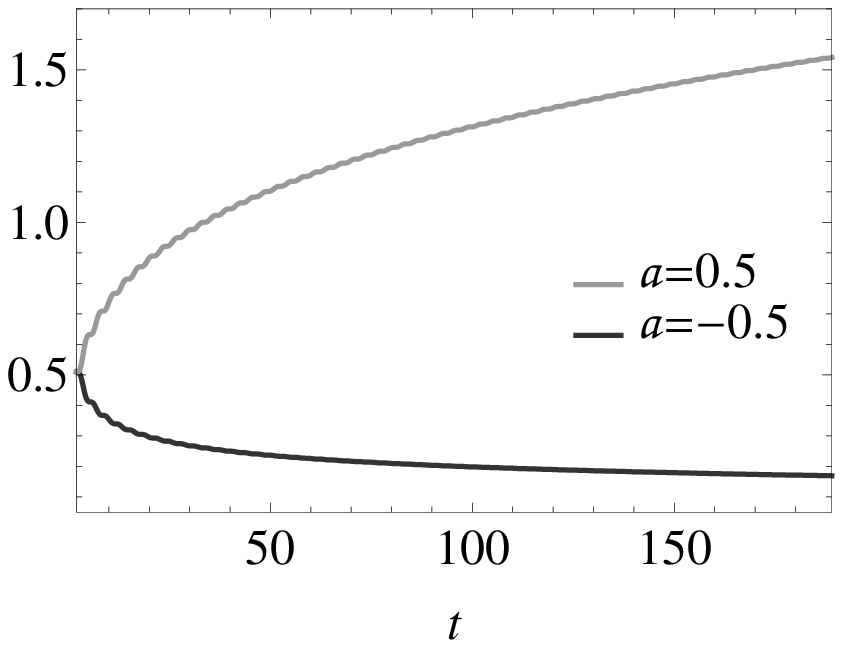}
}
\hspace{1ex}
 \subfigure[$c=1$, $p=q=2$]{
\includegraphics[width=0.3\linewidth]{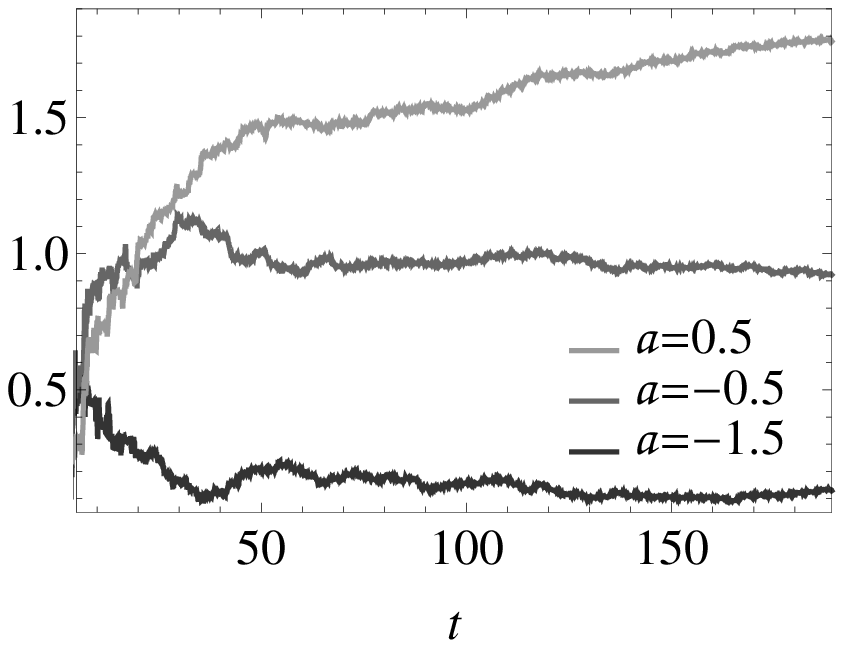}
}
\hspace{1ex}
\subfigure[$c=1$, $p=1$, $q=2$]{
\includegraphics[width=0.3\linewidth]{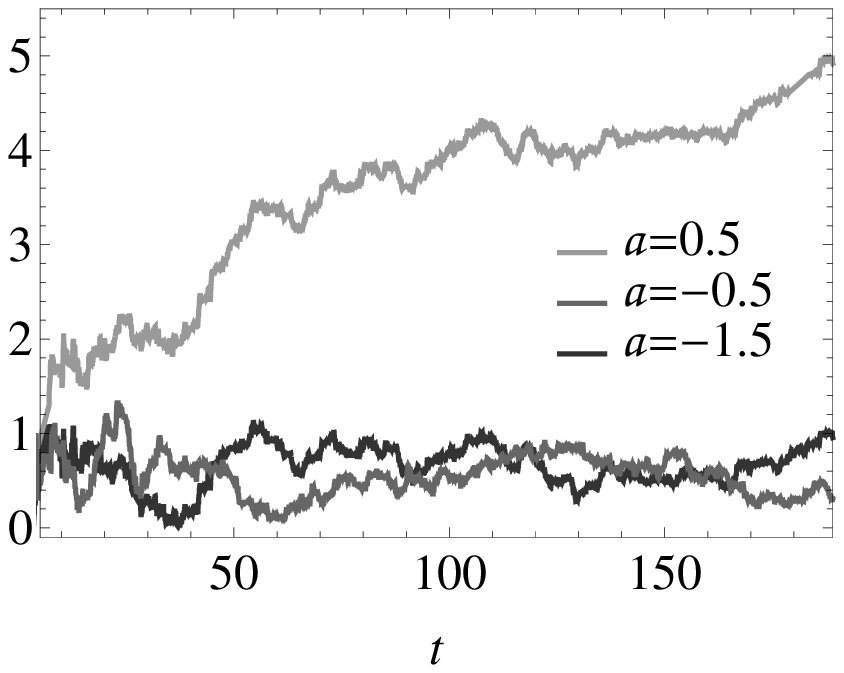}
}
\caption{\small The evolution of $|{\bf z}(t)|$ for sample paths of solutions to system \eqref{ex0} with different values of the parameters.} \label{Fig12}
\end{figure}

In the general case, nonlinear terms of asymptotically autonomous systems can significantly affect the long-term behaviour of solutions~\cite{OS20arxiv}. The goal of this paper is to reveal the role of damped and nonlinear terms in asymptotic regimes of solutions to system \eqref{FulSys} near the equilibrium of the limiting system.

\section{Main results}\label{MainRes}
Let ${\bf z}_\ast(t)\equiv (x_\ast(t,E),y_\ast(t,E))^T$ be a $T(E)$-periodic solution of the limiting system \eqref{LimSys} such that $H_0(x_\ast(t,E),y_\ast(t,E))\equiv E$, $x_\ast(0,E)>0$ and $y_\ast(0,E)= 0$ for all $E\in[0,E_0]$. For any $\sigma_0\in (0,E_0)$, define the domain
\begin{gather*}
\mathcal D(\sigma_0,E_0)=\{(x,y)\in\mathcal B_r: \sigma_0 \leq H_0(x,y)\leq E_0\}.
\end{gather*}
Define 
\begin{gather*}
\mu_{2p}:= \frac{1}{2}{\hbox{\rm tr}} \left({\bf A}_{p}^T(0){\bf A}_{p}(0)\right)= \frac{1}{2}\sum_{i,j=1}^2 \left(\alpha^{p}_{i,j}(0,0)\right)^2>0.
\end{gather*}
Then we have the following:
\begin{Th}\label{Th1}
Let system \eqref{FulSys} satisfy assumptions \eqref{zero}, \eqref{lip}, \eqref{H0as}, \eqref{HFBas},  and \eqref{Azero}. Then for all $\sigma_0\in (0,E_0)$ and $N\in [2p,4p]$ there exist $t_\ast\geq s$ and the chain of transformations $(x,y)\to (E,\varphi)\to (v,\varphi)$,
\begin{gather}
\label{exch1} x(t)= x_\ast\left(\frac{\varphi(t)}{\nu(E(t))},E(t)\right), \quad y(t)= y_\ast\left(\frac{\varphi(t)}{\nu(E(t))},E(t)\right), \\
\label{exch11} v(t)=V_N (E(t),\varphi(t),t), \quad V_N (E,\varphi,t)\equiv E+\sum_{k=1}^{N} t^{-\frac{k}{q}} v_k(E,\varphi)
\end{gather}
such that for all $(x,y)\in \mathcal D(\sigma_0,E_0)$ and $t\geq t_\ast$ system \eqref{FulSys} can be transformed into
\begin{gather}\label{Veq}
\begin{split}
dv=&\left(\Lambda(v,t)+F(v,\varphi,t)\right) \,dt+\sum_{j=1}^2 \tilde \beta_{1,j}(v,\varphi,t) \, dw_j(t),\\
{d\varphi} =&\left(\nu(v)+ G(v,\varphi,t)\right)dt+ \sum_{j=1}^2 \tilde \beta_{2,j}(v,\varphi,t) dw_j(t),
\end{split}
\end{gather}
where $v_k(E,\varphi)$, $F (v,\varphi,t)$, $G (v,\varphi,t)$ and $\tilde \beta_{i,j}(v,\varphi,t)$ are $2\pi$-periodic in $\varphi$, and
\begin{gather}
\nonumber 
	\Lambda(v,t)\equiv \sum_{k=1}^{N} t^{-\frac{k}{q}}\Lambda_k(v), \\ 
\nonumber
	\Lambda_k(E)=
\begin{cases} 
\mathcal O(E), & 1\leq k<2p\\
\mu_{2p}+\mathcal O(E), & k=2p\\
\mathcal O(1), & 2p< k\leq N
\end{cases}, \quad
 v_k(E,\varphi)=
\begin{cases} 
\mathcal O(E), & 1\leq k<2p\\
\mathcal O(1), & 2p\leq k\leq N
\end{cases}
\end{gather}
as $E\to 0$ uniformly for all $\varphi\in\mathbb R$. Moreover, there exist $\sigma_\ast>\sigma_0$ and $v_\ast\in (\sigma_\ast,E_0)$ such that 
\begin{gather*}
F(v,\varphi,t)=\mathcal O(t^{-\frac{N+1}{q}}), \quad G(v,\varphi,t)=\mathcal O(t^{-\frac{1}{q}}), \quad \tilde \beta_{i,j}(v,\varphi,t)=\mathcal O(t^{-\frac{1}{q}}) 
\end{gather*}
as $t\to\infty$ uniformly for all $v\in [\sigma_\ast,v_\ast]$ and $\varphi\in\mathbb R$.
\end{Th}
The proof is contained in Section~\ref{Sec3}.

Note that the chain of transformations described in Theorem~\ref{Th1} can set some coefficients $\Lambda_k(v)$ to zero for $1\leq k<2p$. Let $n\leq 2p$ be the smallest number such that 
\begin{gather}\label{ass1}
\Lambda_k(v)\equiv 0, \quad k<n, \quad \Lambda_n(v)\not\equiv 0.
\end{gather}
Consider the following three cases, depending on the values of $(n,p,q)$:
\begin{eqnarray*}
\Sigma_1&=&\{(n,p,q)\in\mathbb Z^3_+: \quad n<2p, \quad n\leq q\},\\
\Sigma_2&=&\{(n,p,q)\in\mathbb Z^3_+: \quad n=2p\},\\
\Sigma_3&=& \{(n,p,q)\in\mathbb Z^3_+: \quad q<n\leq 2p\}.
\end{eqnarray*}

It follows from Theorem~\ref{Th1} that the behaviour of solutions to system \eqref{FulSys} near the equilibrium $(0,0)$ of the limiting system \eqref{LimSys} is determined by solutions of the transformed system \eqref{Veq} with $v(t)\in (0,E_0)$. It is clear that a long-term behaviour of $v(t)$, starting close to zero, depends on the structure of the right-hand side of the first equation in system \eqref{Veq}. Note that even for similar deterministic systems without noise, the key role can be played by both primary and secondary terms of the asymptotics as $t\to\infty$ (see, for example,~\cite{OS21DCDS}). With this in mind, we consider the following typical assumptions about the structure of these terms:
\begin{align}
\label{ass21}   & 
	\Lambda_n(v)= v\left(\lambda_{n} +\mathcal O(v)\right),\quad \Lambda_{k}(v)=\mathcal O(v), \quad k<2p;\\
\label{ass22}     &\exists\,m,d\in\mathbb Z_+: \quad m\geq 2,\quad
	\Lambda_k(v)=
	\begin{cases}
		 v^m\left(\lambda_{n,m}+\mathcal O(v)\right), &   k=n\\
		\mathcal O(v^m), & n<k<n+d \\ 
		 v\left(\lambda_{n+d}+\mathcal O(v)\right), & k=n+d\\
		\mathcal O(v), & n+d<k<2p
 \end{cases}
\end{align}
as $v\to 0$, where $\lambda_{n},\lambda_{n,m},\lambda_{n+d}={\hbox{\rm const}}\neq 0$. Assumption \eqref{ass21} corresponds to the case, when the primary asymptotic term of $\Lambda(v,t)$ has non-zero linear part. Assumption \eqref{ass22} corresponds to the cases when the leading terms of $\Lambda(v,t)$ are non-linear as $v\to 0$.


Define the functions
\begin{align*}
&\rho({\bf z},t;\xi,\vartheta)\equiv t^{\vartheta} H_0(x,y)-\xi, &\quad& Q_1(\zeta)\equiv \lambda_{n,m}\zeta^m + (\lambda_{n+d}+\delta_{n+d,q}\vartheta_1)\zeta, \\
&Q_2(\zeta)\equiv \lambda_{n,m}\zeta^m + \mu_{2p}, &\quad &
Q_3(\zeta)\equiv \lambda_{n,m}\zeta^m + (\lambda_{n+d}+\delta_{n+d,q}\vartheta_2)\zeta+\mu_{2p},
\end{align*}
the parameters
\begin{align*}
 \xi_0=&\frac{\mu_{2p}}{|\lambda_{n}+\delta_{n,q}\vartheta_0|}, \quad &\xi_1=&\left|\frac{\lambda_{n+d}+\delta_{n+d,q}\vartheta_1}{\lambda_{n,m}}\right|^{\frac{1}{m-1}},\quad &\xi_2=&\left(\frac{\mu_{2p}}{|\lambda_{n,m}|}\right)^{\frac{1}{m}},\\
 \vartheta_0=&\frac{2p-n}{q}>0, \quad &\vartheta_1=&\frac{d}{q(m-1)}>0, \quad &\vartheta_2=&\frac{2p-n}{qm}>0,
\end{align*}
and $\kappa_j=\min\left\{q^{-1},\vartheta_j\right\}>0$ for $j\in\{0,1,2\}$.
Consider a reduced equation in the form
\begin{gather}\label{ueq1}
		 \frac{du}{dt}=\Lambda(u,t),\quad t\geq t_\ast>0, \quad  \Lambda(u,t)\equiv \sum_{k=n}^{N}t^{-\frac{k}{q}}\Lambda_k(u)
\end{gather}
with some $N\in[2p,4p]$. Then, we have the following:
\begin{Lem}\label{Lem1}
Let $(n,p,q)\in\Sigma_1$ and assumption \eqref{ass21} hold with $\lambda_{n}+\delta_{n,q}\vartheta_0<0$. Then there exists a particular solution $u_0(t)=t^{-\vartheta_0} \zeta_0(t)$ of equation \eqref{ueq1} such that 
\begin{gather}\label{uastas}
\zeta_0(t)= \xi_0+ 
\begin{cases} 
\mathcal O(1), & n<q,\\
\mathcal O\big(t^{- (1-\epsilon)\varkappa_0 }\big), \quad \forall\,\epsilon\in (0,1), & n=q,
\end{cases}
\end{gather}
as $t\to\infty$, where
$ 
\varkappa_0=\min\{\kappa_0, |\lambda_n+\vartheta_0|\}$.
\end{Lem}

The following theorem shows that the solution $u_0(t)$ of the reduced equation is stable in the full stochastic system. 
\begin{Th}\label{Th2}
Let system \eqref{FulSys} satisfy \eqref{zero}, \eqref{lip}, \eqref{H0as}, \eqref{HFBas}, \eqref{Azero} and $(n,p,q)\in\Sigma_1$ be integers such that assumptions \eqref{ass1} and \eqref{ass21} hold. 
 If $\lambda_{n}+\delta_{n,q}\vartheta_0<0$, then for all $\varepsilon_1>0$ and $\varepsilon_2>0$ there exist $\delta_0>0$ and $t_0>0$ such that the solution ${\bf z}(t)$ of system \eqref{FulSys}
with initial data ${\bf z}(t_0)={\bf z}_0$, $|\rho({\bf z}_0,t_0;\zeta_0(t_0),\vartheta_0)|< \delta_0$, satisfies
\begin{gather}\label{defineq2}
\mathbb P\left(\sup_{ t\geq t_0 }  |\rho({\bf z}(t),t;\zeta_0(t),\vartheta_0)|>\varepsilon_1\right)<\varepsilon_2.
\end{gather}
\end{Th}
The proof is contained in Section~\ref{Sec4}.

Under assumption \eqref{ass22}, consider the following three cases:
\begin{align*}
\text{{\bf Case I}:}\quad &   m>\frac{2p-n}{2p-(n+d)}, \quad  \lambda_{n,m}<0, \quad \lambda_{n+d}+\delta_{n+d,q}\vartheta_1>0;\\
\text{{\bf Case II}:} \quad &  m<\frac{2p-n}{2p-(n+d)}, \quad  \lambda_{n,m}<0;\\
\text{{\bf Case III}:}\quad  &  m=\frac{2p-n}{2p-(n+d)}, \quad  \lambda_{n,m}<0, \\
 & \text{or}\\
& m=\frac{2p-n}{2p-(n+d)}, \quad  \lambda_{n,m}>0, \quad \lambda_{n+d}+\delta_{n+d,q}\vartheta_2<0, \quad \mu_{2p}<\lambda_{n,m} (m-1)K,
\end{align*}
where 
\begin{gather*}
K=\left|\frac{\lambda_{n+d}+\delta_{n+d,q}\vartheta_2}{m \lambda_{n,m}}\right|^{\frac{m}{m-1}}.
\end{gather*}
Then, we have the following:
\begin{Lem}\label{Lem3}
Let $(n+d,p,q)\in\Sigma_1$ and assumption \eqref{ass22} hold. Then,
\begin{itemize}
\item in {\bf Case I}, there exists a particular solution 
$u_1(t)=t^{-\vartheta_1} \zeta_1(t)$ of equation \eqref{ueq1} such that 
\begin{gather}\label{u1as}
\zeta_1(t)= \xi_1+\begin{cases} \mathcal O(1), & n+d<q,\\
\mathcal O\big(t^{- (1-\epsilon)\varkappa_1 }\big), \quad \forall\,\epsilon\in (0,1), & n+d=q,
\end{cases}
\end{gather}
as $t\to\infty$, where $\varkappa_1= \min\{\kappa_1, |Q_1'(\xi_1)|\}$;
\item in {\bf Case II}, there exists a particular solution $u_2(t)=t^{-\vartheta_2} \zeta_2(t)$ of equation \eqref{ueq1} such that  
\begin{gather}\label{u2as}
\zeta_2(t)=\xi_{2}+\mathcal O (1), \quad t\to\infty;
\end{gather}
\item in {\bf Case III}, there exist $\xi_3>0$ such that $Q_3(\xi_3)=0$, $Q_3'(\xi_3)<0$ and a particular solution $u_3(t)=t^{-\vartheta_2}\zeta_3(t)$ of equation \eqref{ueq1} such that  
\begin{gather}\label{u3as}
\zeta_3(t)=\xi_3+\begin{cases}\mathcal O(1), & n+d<q,\\
\mathcal O\big(t^{-  (1-\epsilon)\varkappa_3 }\big), \quad \forall\,\epsilon\in (0,1), & n+d=q,
\end{cases}
\end{gather}
as $t\to\infty$, where $\varkappa_3= \min\{\kappa_2, |Q_3'(\xi_3)|\}$.
\end{itemize}
\end{Lem}

As in the case of assumption \eqref{ass21}, the following theorem shows that the solutions $u_1(t)$, $u_2(t)$ and $u_3(t)$ are stable in the stochastic system \eqref{Veq}.
\begin{Th}\label{Th3}
Let system \eqref{FulSys} satisfy \eqref{zero}, \eqref{lip}, \eqref{H0as}, \eqref{HFBas}, \eqref{Azero} and $(n+d,p,q)\in\Sigma_1$ be integers such that assumptions \eqref{ass1} and \eqref{ass22} hold. Then,
\begin{itemize}
\item  in {\bf Case I}, for all $\varepsilon_1>0$ and $\varepsilon_2>0$ there exist $\delta_0>0$ and $t_0>0$ such that the solution ${\bf z}(t)$ of system \eqref{FulSys}
with initial data ${\bf z}(t_0)={\bf z}_0$, $|\rho({\bf z}_0,t_0;\zeta_1(t_0),\vartheta_1)|< \delta_0$, satisfies
\begin{gather*} 
\mathbb P\left(\sup_{ t\geq t_0 }  |\rho({\bf z}(t),t;\zeta_1(t),\vartheta_1)|>\varepsilon_1\right)<\varepsilon_2.
\end{gather*}
\item  in {\bf Case II}, for all $\varepsilon_1>0$ and $\varepsilon_2>0$ there exist $\delta_0>0$ and $t_0>0$ such that the solution ${\bf z}(t)$ of system \eqref{FulSys}
with initial data ${\bf z}(t_0)={\bf z}_0$, $|\rho({\bf z}_0,t_0;\zeta_2(t_0),\vartheta_2)|< \delta_0$, satisfies
\begin{gather*} 
\mathbb P\left(\sup_{ t\geq t_0 }  |\rho({\bf z}(t),t;\zeta_2(t),\vartheta_2)|>\varepsilon_1\right)<\varepsilon_2.
\end{gather*}
\item  in {\bf Case III}, for all $\varepsilon_1>0$ and $\varepsilon_2>0$ there exist $\delta_0>0$ and $t_0>0$ such that the solution ${\bf z}(t)$ of system \eqref{FulSys}
with initial data ${\bf z}(t_0)={\bf z}_0$, $|\rho({\bf z}_0,t_0;\zeta_3(t_0),\vartheta_2)|< \delta_0$, satisfies
\begin{gather*} 
\mathbb P\left(\sup_{ t\geq t_0 }  |\rho({\bf z}(t),t;\zeta_3(t),\vartheta_2)|>\varepsilon_1\right)<\varepsilon_2.
\end{gather*}
\end{itemize}
\end{Th}
The proof is contained in Section~\ref{Sec5}.
 
If $(n,p,q)\in \Sigma_2$, the long-term behaviour of solutions is determined by the zeros of the coefficient $\Lambda_n(v)$. Assume that
\begin{align}
\label{ass3}
&\exists \, \xi_\ast\in (0,E_0): \quad \Lambda_n(\xi_\ast)=0, \quad \Lambda_n'(\xi_\ast)\neq 0,
\end{align}
and consider an additional assumption on the intensity of stochastic perturbations:
\begin{gather}
\label{asslimc} \exists\, \mu > 0 : \quad  \frac{1}{2}|{\hbox{\rm tr}}({\bf A}^T{\bf A})| \leq  \mu t^{-\frac{2p}{q}}
\end{gather}
for all $(x, y)\in \mathcal B_r$ and $t\geq s$. Let us show that, in this case, damped perturbations can lead to the appearance of states close to periodic solutions of the limiting system. In particular, we have the following:
\begin{Lem}\label{Lem4}
Let $(n,p,q)\in\Sigma_2$ and assumption \eqref{ass3} hold. Then, there exists a particular solution $u_\ast(t)$ of equation \eqref{ueq1} such that 
\begin{gather}\label{uastasc}
u_\ast(t)=  \xi_\ast+\begin{cases} \mathcal O(1), & n=2p<q, \\
\mathcal O\big(t^{- (1-\epsilon)\varkappa_\ast}\big), \quad \forall\, \epsilon\in (0,1), & n=2p\geq q, 
\end{cases}
\end{gather}
as $t\to\infty$, where
$\varkappa_\ast= \displaystyle  \min\left\{ q^{-1},|\Lambda_{n}'(\xi_\ast)|\right\}$.
\end{Lem}
In this case, it can be shown that the solution $u_\ast(t)$ is stochastically stable in system \eqref{FulSys} on a finite but asymptotically long time interval as $\mu\to 0$. 
\begin{Th}\label{Th4}
Let system \eqref{FulSys} satisfy \eqref{zero}, \eqref{lip}, \eqref{H0as}, \eqref{HFBas}, \eqref{Azero}, \eqref{asslimc} and $(n,p,q)\in\Sigma_2$ be integers such that assumptions \eqref{ass1} and \eqref{ass3} hold. 
 If $\Lambda_n'(\xi_\ast)<0$, then for all $\varepsilon_1>0$ and $\varepsilon_2>0$ there exist $\delta_0>0$ and $t_0>0$ such that the solution ${\bf z}(t)$ of system \eqref{FulSys}
with initial data ${\bf z}(t_0)={\bf z}_0$, $|\rho({\bf z}_0,t_0;u_\ast(t_0),0)|< \delta_0$, satisfies
\begin{gather}\label{defineq6}
\mathbb P\left(\sup_{ 0\leq t-t_0\leq \mathcal T }  |\rho({\bf z}(t),t;u_\ast(t),0)|>\varepsilon_1\right)<\varepsilon_2,
\end{gather}
where $\mathcal T=C\mu^{-1}\delta_0^2 t_0^{2p/q}$ if $2p<q$, $\mathcal T=t_0 (e^{C\mu^{-1}\delta_0^2}-1)$ if $2p=q$ and $\mathcal T=\infty$ if $2p>q$ with some $C={\hbox{\rm const}}>0$.
\end{Th}
The proof is contained in Section~\ref{Sec6}.

Finally, if $(n,p,q)\in\Sigma_3$, then $\Lambda(v,t)$ decays fast enough and the solutions of the perturbed system can tend to any $\xi_\ast\in (0,E_0)$ such that  
\begin{align}
\label{ass4}
& \Lambda_n(\xi_\ast)\neq 0, \quad \Lambda_n'(\xi_\ast)\neq 0.
\end{align}
In this case, we have the following:
\begin{Lem}\label{Lem5}
Let $(n,p,q)\in\Sigma_3$ and assumption \eqref{ass3} hold. Then, for all $\xi_\ast\in (0,E_0)$ such that $\Lambda_n(\xi_\ast)\neq 0$ and $\Lambda_n'(\xi_\ast)< 0$ there exists a particular solution $v_\ast(t)$ of equation \eqref{ueq1} such that 
\begin{gather}\label{uasast}
v_\ast(t)= \left(\xi_\ast+\mathcal O\big(t^{- \frac{1}{q}}\big)\right), \quad t\to\infty.
\end{gather}
\end{Lem} 
The solution $v_\ast(t)$ turns out to be stable in the full stochastic system \eqref{FulSys}.
 \begin{Th}\label{Th5}
Let system \eqref{FulSys} satisfy \eqref{zero}, \eqref{lip}, \eqref{H0as}, \eqref{HFBas}, \eqref{Azero}, \eqref{asslimc}, $(n,p,q)\in\Sigma_3$ and $\xi_\ast\in (0,E_0)$ such that assumptions \eqref{ass1} and \eqref{ass4} hold.
 If $\Lambda_n'(\xi_\ast)<0$, then for all $\varepsilon_1>0$ and $\varepsilon_2>0$ there exist $\delta_0>0$ and $t_0>0$ such that the solution ${\bf z}(t)$ of system \eqref{FulSys}
with initial data ${\bf z}(t_0)={\bf z}_0$, $|\rho({\bf z}_0,t_0;v_\ast(t_0),0)|< \delta_0$, satisfies
\begin{gather*}
\mathbb P\left(\sup_{ t\geq t_0 }  |\rho({\bf z}(t),t;v_\ast(t),0)|>\varepsilon_1\right)<\varepsilon_2.
\end{gather*}
\end{Th}
The proof is contained in Section~\ref{Sec7}.

Let us note that Theorems~\ref{Th2} and \ref{Th3} describe the conditions that ensure the polynomial convergence of the solutions of perturbed system \eqref{FulSys} to the equilibrium $(0,0)$ of the limiting system \eqref{LimSys}. Theorem~\ref{Th4} describes the appearance of stochastically stable cycles in the perturbed system. Finally, in the case of Theorem~\ref{Th5}, the perturbed system behaves qualitatively like the corresponding limiting system. 

\section{Changes of variables}\label{Sec3}
In this section, we construct variables transformations that reduce system \eqref{FulSys} to the form \eqref{Veq}. 

\subsection{Energy-angle variables}
Define auxiliary functions
\begin{gather*}
X(\varphi,E)\equiv  x_\ast\left(\frac{\varphi}{\nu(E)},E\right), \quad  Y(\varphi,E)\equiv  y_\ast\left(\frac{\varphi}{\nu(E)},E\right)
\end{gather*} 
that are $2\pi$-periodic in $\varphi$ and satisfy the following autonomous system:
\begin{gather}\label{LimSys2pi}
    \nu(E)\frac{\partial X}{\partial \varphi}=\partial_Y H_0(X,Y), \quad
    \nu(E)\frac{\partial Y}{\partial \varphi}=-\partial_X H_0(X,Y).
\end{gather}
We use these functions to write system \eqref{FulSys} in the energy-angle variables $(E,\varphi)$.
Note that differentiating the identity $H_0(X(\varphi,E),Y(\varphi,E))\equiv E$ with respect to $E$ yields
\begin{gather*}
\det\frac{\partial(X,Y)}{\partial (\varphi,E)}=\begin{vmatrix}
        \partial_\varphi X & \partial_E X\\
        \partial_\varphi Y& \partial_E Y
    \end{vmatrix} \equiv  \frac{1}{\nu(E)}\neq 0, \quad E\in [0,E_0].
\end{gather*}
Hence, the mapping $(x,y)\mapsto (E,\varphi)$ is invertible for all $E\in [0,E_0]$ and $\varphi\in[0,2\pi)$. Denote by 
\begin{gather*}
	E=I(x,y),  \quad \varphi=\Phi(x,y)
\end{gather*}
the inverse transformation to \eqref{exch1}, and for any smooth function $U({\bf z},t)$, define the following operator associated with the system of stochastic differential equations \eqref{FulSys}:
\begin{gather*}
L U :=  \partial_t U  + \left(\nabla_{\bf z} U\right)^T {\bf a}+\frac{1}{2}{\hbox{\rm tr}}\left({\bf A}^T {\bf H}_{\bf z}(U){\bf A}\right),
\end{gather*} 
where  
\begin{gather*}
\nabla_{\bf z} U \equiv \begin{pmatrix}\partial_x U  \\  \partial_y U \end{pmatrix}, \quad 
{\bf H}_{\bf z}(U)\equiv 	
\begin{pmatrix}
		\partial_x^2 U & \partial_x\partial_y U \\
		\partial_y\partial_x U & \partial_y^2 U
\end{pmatrix}.
\end{gather*}
Note that $L$ plays an important role in the study of stochastic differential equations (see, for example,~\cite[\S 3.3]{RH12}). In the new variables $\hat{\bf z}=(E,\varphi)^T$ system \eqref{FulSys} takes the following form:
\begin{equation}
	\label{FulSys2}
	d \hat {\bf z} =  {\bf b}(\hat{\bf z},t) dt +  {\bf B}(\hat{\bf z},t)\,d{\bf w}(t),
\end{equation}
with the coefficients
\begin{gather*}
	{\bf b}(\hat{\bf z},t)\equiv \begin{pmatrix} L I \\ L \Phi\end{pmatrix}\Big|_\eqref{exch1}, \quad
	{\bf B}(\hat{\bf z},t)=\begin{pmatrix}
		\partial_x I & \partial_y I \\
		\partial_x \Phi & \partial_y \Phi
\end{pmatrix} {\bf A}({\bf z},t)\Big|_\eqref{exch1}
\end{gather*}
that are $2\pi$-periodic in $\varphi$. Moreover, it follows from \eqref{HFBas} that
\begin{gather}\label{bBas}
{\bf b}(\hat{\bf z},t)=\sum_{k=0}^\infty t^{-\frac{k}{q}} {\bf b}_k(E,\varphi), \quad {\bf B}({\bf z},t)=\sum_{k=1}^\infty t^{-\frac{k}{q}} {\bf B}_{k}(E,\varphi),  \quad t\to\infty,
\end{gather}
where
\begin{gather}
\label{b0bk}
\begin{split}
& {\bf b}_0(E,\varphi)\equiv\begin{pmatrix}  0 \\ \nu(E)\end{pmatrix}, \quad 
{\bf b}_k(E,\varphi)\equiv\begin{pmatrix}  f_k(E,\varphi) \\ g_k(E,\varphi)\end{pmatrix}, \quad {\bf B}_k(E,\varphi)\equiv\{\beta^k_{i,j}(E,\varphi)\}_{2\times 2}, \\ 
& f_k(E,\varphi)\equiv 
	\displaystyle \left(\nabla_{\bf z} I\right)^T {\bf a}_k+\frac{1}{2}\sum_{i+j=k}{\hbox{\rm tr}}\left({\bf A}_i^T {\bf H}_{\bf z}(I){\bf A}_j\right), \quad  \left(\beta_{1,1}^k, \beta_{1,2}^k\right)\equiv  \left(\nabla_{\bf z} I\right)^T {\bf A}_k,\\ 
& g_k(E,\varphi)\equiv 
	\displaystyle \left(\nabla_{\bf z} \Phi\right)^T {\bf a}_k+\frac{1}{2}\sum_{i+j=k}{\hbox{\rm tr}}\left({\bf A}_i^T {\bf H}_{\bf z}(\Phi){\bf A}_j\right), \quad \left(\beta_{2,1}^k, \beta_{2,2}^k\right)\equiv  \left(\nabla_{\bf z} \Phi\right)^T {\bf A}_k.
\end{split}
\end{gather}
It can easily be checked that 
\begin{gather}\label{deriv}
\begin{split}
  \partial_x \begin{pmatrix} I \\  \Phi \end{pmatrix} \Big|_{\eqref{exch1}} \equiv \begin{pmatrix} -\nu \partial_\varphi Y \\  \nu\partial_E Y \end{pmatrix},\quad \partial^2_x \begin{pmatrix} I \\  \Phi \end{pmatrix}\Big|_{\eqref{exch1}} \equiv \nu (\partial_E Y \partial_\varphi-\partial_\varphi Y\partial_E) \begin{pmatrix} -\nu \partial_\varphi Y \\  \nu\partial_E Y \end{pmatrix},\\
	\partial_y \begin{pmatrix} I \\  \Phi \end{pmatrix} \Big|_{\eqref{exch1}} \equiv \begin{pmatrix}  \nu\partial_\varphi X \\  - \nu\partial_E X \end{pmatrix},\quad 
\partial^2_y \begin{pmatrix} I \\  \Phi \end{pmatrix}\Big|_{\eqref{exch1}} \equiv \nu (\partial_\varphi X \partial_E-\partial_E X\partial_\varphi) \begin{pmatrix} \nu \partial_\varphi X \\  -\nu\partial_E X \end{pmatrix},\\
\partial_x\partial_y \begin{pmatrix} I \\  \Phi \end{pmatrix}\Big|_{\eqref{exch1}} \equiv \nu (\partial_\varphi X \partial_E-\partial_E X \partial_\varphi) \begin{pmatrix} -\nu \partial_\varphi Y \\  \nu\partial_E Y \end{pmatrix}.
\end{split}
\end{gather}
From \eqref{H0as} and \eqref{LimSys2pi} it follows that $X(\varphi,E)=\sqrt{2E}\cos \varphi+\mathcal O(E)$ and $Y(\varphi,E)=-\sqrt{2E}\sin\varphi+\mathcal O(E)$ as $E\to 0$ uniformly for all $\varphi\in\mathbb R$. Hence,
\begin{gather*}
{\bf H}_{\bf z}(I)=\begin{pmatrix} 1 & 0 \\ 0 & 1 \end{pmatrix}+\mathcal O(E^{\frac{1}{2}}), \quad 
{\bf H}_{\bf z}(\Phi)=\frac{1}{2E}\begin{pmatrix} \sin(2\varphi) & \cos(2\varphi) \\ \cos(2\varphi) & -\sin(2\varphi)   \end{pmatrix}+\mathcal O(E^{-\frac 12}), \quad E\to 0.
\end{gather*}
Combining this with \eqref{b0bk} and \eqref{deriv}, we obtain 
\begin{align*} 
  f_k(E,\varphi)&=
	\begin{cases}
		\mathcal O(E),&k<2p\\
		\mathcal O(1),&k\geq 2p
	\end{cases},
	\quad  
&g_k(E,\varphi)&=
	\begin{cases}
		\mathcal O(1),&k<2p\\
		\mathcal O(E^{-1}),&k\geq 2p
	\end{cases}, \\ 
 \beta^k_{1,j}(E,\varphi)&=
	\begin{cases}
		\mathcal O(E),&k<2p\\
		\mathcal O(E^{\frac 12}),&k\geq 2p
	\end{cases}, \quad 
&\beta^k_{2,j}(E,\varphi)&=
\begin{cases}
		\mathcal O(1),&k<2p\\
		\mathcal O(E^{-\frac 12}),&k\geq 2p
	\end{cases}, \quad j\in\{1,2\}
\end{align*}
as $E\to 0$ uniformly for all $\varphi\in\mathbb R$. In particular, 
\begin{gather}\label{f2p}
f_{2p}(E,\varphi)=\mu_{2p}+\mathcal O (E^{\frac 12}), \quad E\to 0.
\end{gather}
\subsection{Averaging}
From \eqref{FulSys2} and \eqref{b0bk} it follows that $\varphi(t)$ changes rapidly in comparison to potential variations of $E(t)$ for large values of $t$. Therefore, for further simplification of the system, we average the drift term in the first equation of system \eqref{FulSys2} over $\varphi$. Note that this trick is usually used in perturbation theory (see, for example,~\cite{BM61,AN84,AKN06}).

Consider a near-identity transformation of system \eqref{FulSys2} in the form \eqref{exch11}. The functions $v_k(E,\varphi)$ are sought in such a way that the drift term of the equation for the new variable $v(t) \equiv  V_N(E(t),\varphi(t),t)$ does not depend on $\varphi$ at least in the first $N\in [2p,4p]$ terms of its asymptotics as $t\to\infty$.  Applying It\^{o}'s formula to $V_N(E,\varphi,t)$ yields 
\begin{gather}\label{dVN}
dv=\mathcal L V_N dt+  \left(\nabla_{\tilde{\bf z}} V_N\right)^T {\bf B} \, d{\bf w}(t), 
\end{gather}
where
\begin{gather*}
\mathcal L V_N \equiv  \partial_t V_N  + \left(\nabla_{\tilde {\bf z}} V_N\right)^T {\bf b}+\frac{1}{2}{\hbox{\rm tr}}\left({\bf B}^T {\bf H}_{\tilde{\bf z}}(V_N){\bf B}\right).
\end{gather*}
Taking into account \eqref{bBas} and \eqref{exch11}, we obtain the following asymptotic expansion:
\begin{equation*}
\begin{split}
\mathcal LV_N=&\sum_{k=1}^\infty t^{-\frac{k}{q}} \Big(f_k+\nu \partial_\varphi v_k-\frac{k-q}{q}v_{k-q}\Big)+\sum_{k=2}^\infty t^{-\frac{k}{q}} \sum_{i+j=k} \left(\nabla_{\tilde {\bf z}} v_{i}\right)^T {\bf b}_{j}\\
 & + \frac{1}{2}\sum_{k=3}^\infty t^{-\frac{k}{q}} \sum_{i+j+m=k} {\hbox{\rm tr}}\left({\bf B}^T_{i} {\bf H}_{\tilde{\bf z}}(v_{j}){\bf B}_{m}\right) 
\end{split}
\end{equation*}
as $t\to\infty$, where it is assumed that $f_k\equiv g_k\equiv0$,  ${\bf B}_{k} \equiv0 $ if $k<1$ and $v_j\equiv 0$ if $j<1$ or $j>N$. 
A comparison of \eqref{dVN} with \eqref{Veq} gives the following chain of differential equations for determining the coefficients $v_k(E,\varphi)$:
\begin{gather}\label{vkeq}
\nu(E)\partial_\varphi v_k=\Lambda_k(E)-f_k(E,\varphi)-R_k(E,\varphi), \quad 1\leq k\leq N,
\end{gather}
where each function $R_k(E,\varphi)$ is explicitly expressed in terms of $v_1(E,\varphi),\dots,v_{k-1}(E,\varphi)$. In particular, $R_1\equiv 0$,
\begin{eqnarray*}
R_2&\equiv & \left(\nabla_{\tilde {\bf z}} v_{1}\right)^T {\bf b}_{1}-v_1\partial_E\Lambda_1 -\frac{2-q}{q}v_{2-q},\\
R_3&\equiv & \sum_{i+j=3}\Big(\left(\nabla_{\tilde {\bf z}} v_{i}\right)^T {\bf b}_{j}-v_{i}\partial_E\Lambda_{j}\Big)-\frac{v_1^2}{2}\partial_E^2 \Lambda_1-\frac{3-q}{q}v_{3-q}+\frac{1}{2}{\hbox{\rm tr}}\left({\bf B}^T_{1} {\bf H}_{\tilde{\bf z}}(v_{1}){\bf B}_{1}\right),\\
R_k&\equiv & \sum_{i+j=k}\left(\nabla_{\tilde {\bf z}} v_{i}\right)^T {\bf b}_{j}-\frac{k-q}{q}v_{k-q} +\frac{1}{2}\sum_{i+j+m=k}{\hbox{\rm tr}}\left({\bf B}^T_{i} {\bf H}_{\tilde{\bf z}}(v_{j}){\bf B}_{m}\right)\\
& &-\sum_{j+m_1+2m_2+\dots+im_i=k} C_{i,j,m} v_1^{m_1}\cdots v_i^{m_i}\partial_E^{m_1+\dots+m_i} \Lambda_j, \quad  k\leq N,
\end{eqnarray*}
where $C_{i,j,m}={\hbox{\rm const}}$. Define
\begin{gather*}
\Lambda_k(E)=\langle f_k(E,\varphi)+R_k(E,\varphi)\rangle_\varphi, \quad \langle C(E,\varphi)\rangle_\varphi:=\frac{1}{2\pi}\int\limits_0^{2\pi} C(E,\varphi)\,d\varphi.
\end{gather*}
Then the right-hand side of \eqref{vkeq} turns out to be $2\pi$-periodic in $\varphi$ with zero average. 
Integrating 
\begin{gather*}
v_k(E,\varphi)=-\frac{1}{\nu(E)}\int\limits_0^\varphi \{f_k(E,\varsigma)+R_k(E,\varsigma)\}_\varsigma\, d\varsigma+\hat v_k(E),
\end{gather*}
Here $\{C(E,\varsigma)\}_\varsigma=C(E,\varsigma)-\langle C(E,\varsigma)\rangle_\varsigma$, and the functions $\hat v_k(E)$ are chosen such that $\langle v_k(E,\varphi)\rangle_\varphi\equiv 0$. 
It can easily be checked that 
\begin{gather}\label{est}
\begin{split}
& R_k(E,\varphi)=\begin{cases}	\mathcal O(E), & k\leq 2p \\ 	\mathcal O(1), & 2p<k\leq N \end{cases},\quad 
\Lambda_k(E)=
\begin{cases}
	\mathcal O(E),& k<2p\\
	\mathcal O(1),& 2p \leq k\leq N
\end{cases},
\\
&  v_k(E,\varphi)=
\begin{cases}
	\mathcal O(E), & k<2p\\
	\mathcal O(1),&  2p \leq k\leq N
\end{cases} 
\end{split}
\end{gather}
as $E\to 0$ uniformly for all $\varphi\in\mathbb R$. In particular, from \eqref{f2p} it follows that $\Lambda_{2p}(E)=\mu_{2p}+\mathcal O(E)$.
 
From \eqref{exch11} it follows that for all $\sigma_0\in (0,E_0)$ and $\epsilon\in (0, \sigma_0)$ there exists $t_\ast \geq  s $ such that
\begin{gather*}
	|V_N(E,\varphi,t)-E|\leq \epsilon, \quad |\partial_E V_N(E,\varphi,t)-1|\leq \epsilon, \quad |\partial_\varphi V_N(E,\varphi,t)|\leq \epsilon 
\end{gather*}
for all $E\in [\sigma_0,E_0]$, $\varphi\in \mathbb R$ and $t\geq t_\ast$. Hence, the transformation $(E,\varphi,t)\mapsto (v,\varphi,t)$ is invertible for all $v\in [\sigma_\ast,v_\ast]$, $\varphi\in [0,2\pi)$ and $t\geq t_\ast$ with $v_\ast=E_0-\epsilon$ and $\sigma_\ast=\sigma_0+\epsilon$. Denote by $E=\mathcal E(v,\varphi,t)$ the inverse transformation to \eqref{exch11}. Then 
\begin{equation*}
	 \tilde \beta_{1,j}(v,\varphi,t)\equiv \left(\nabla_{\tilde{\bf z}} V_N\right)^T {\bf B}\Big|_{E=\mathcal E(v,\varphi,t)},\quad
	 F(v,\varphi,t) \equiv -\sum_{k=1}^{N} t^{-\frac{k}{q}} \Lambda_k(v)+ \mathcal L V_N(E,\varphi,t)\Big|_{E=\mathcal E(v,\varphi,t)}.
\end{equation*}
In this case, the equation for $\varphi$ does not change significantly under the transformation \eqref{exch11}:
\begin{eqnarray*} 
	G(v,\varphi,t)\equiv \nu(\mathcal E(v,\varphi,t))-\nu(v)+\sum_{k=1}^\infty t^{-\frac{k}{q}} g_k(\mathcal E(v,\varphi,t),\varphi), \quad
	\tilde \beta_{2,j}(v,\varphi,t)\equiv \beta_{2,j}(\mathcal E(v,\varphi,t),\varphi,t).
\end{eqnarray*}
It follows easily that
\begin{align*}
&F(v,\varphi,t)=\mathcal O (t^{-\frac{N+1}{q}} ), \quad 
G(v,\varphi,t)=\mathcal O(t^{-\frac{1}{q}})+\mathcal O(v^{-1})\mathcal O(t^{-\frac{2p}{q}}), \quad 
 \tilde \beta_{i,j}(v,\varphi,t)=\mathcal O(v^{-\frac{1}{2}})\mathcal O(t^{-\frac{1}{q}})
\end{align*}
as $t\to\infty$ and $v\to 0$ uniformly for all $\varphi\in\mathbb R$.

\section{Proof of Theorem~\ref{Th2}}
\label{Sec4}

\begin{proof}[Proof of Lemma~\ref{Lem1}]
Substituting $u(t)=t^{-\vartheta_0}(\xi_0+ \xi(t))$ into equation \eqref{ueq1} yields
\begin{gather}\label{xieq1}
	\frac{d\xi}{dt}=\tilde \Lambda(\xi_0+ \xi,t), \quad \tilde \Lambda(\zeta,t)\equiv t^{\vartheta_0}\Lambda\left(t^{-\vartheta_0}  \zeta,t\right)+\vartheta_0 t^{-1} \zeta .
\end{gather}
We see that $\tilde \Lambda(\xi_0+ \xi,t)=t^{-\frac{n}{q}}\left((\lambda_{n}+\delta_{n,q}\vartheta_0)\xi+ \mathcal O(t^{-\kappa_0})\right)$ as $t\to \infty$ uniformly for all $|\xi|<\infty$. Hence, there exist $C_0>0$, $s_0\geq t_\ast$ and $\Delta_0>0$ such that 
\begin{gather}\label{xiineqlem1}
\frac{d|\xi|}{dt}\leq t^{-\frac{n}{q}} \left(-\gamma_0 |\xi|+C_0 t^{-\kappa_0}\right)
\end{gather}
for all $t\geq s_0$ and $|\xi|\leq \Delta_0$, where $\gamma_0=|\lambda_n+\delta_{n,q}\vartheta_0|>0$. 
For all $\varepsilon\in (0,\Delta_0)$ define
\begin{gather*}
\delta_0=\frac{2 C_0 s_0^{-\kappa_0}}{ \gamma_0}<\varepsilon, \quad t_0=\max\left\{s_0, \left(\frac{4 C_0}{ \varepsilon \gamma_0 }\right)^{\frac{1}{\kappa_0}}\right\}.
\end{gather*}
Then $d|\xi(t)|/dt<0$ for solutions of \eqref{zeq1} such that $\delta_0 \leq |\xi(t)|\leq \varepsilon$ as $t\geq t_0$. Hence, any solution of \eqref{xieq1} with initial data $|\xi(t_0)|\leq \delta_0$ cannot leave the domain $|\xi|\leq \varepsilon$ as $t\geq t_0$.

Integrating this inequality with respect to $t$ as $t\geq t_0$ with some $t_0\geq s_0$, we obtain
\begin{gather*}
  |\xi(t)|\leq |\xi(t_0)| e^{-\gamma_0(\tau(t)-\tau(t_0))}+C_0 \int\limits_{t_0}^t e^{\gamma_0(\tau(\varsigma)-\tau(t))} \tau'(\varsigma) \varsigma^{-\kappa_0}\, d \varsigma, \quad \tau'(t)=t^{-\frac{n}{q}}.
\end{gather*}
It can easily be checked that if $n< q$, then $\gamma_0=|\lambda_n|$ and $\xi(t)=\mathcal O(1)$ as $t\to\infty$.
If $n=q$, then $\gamma_0=|\lambda_n+\vartheta_0|$ and we have the following inequality:
\begin{gather*}
|\xi(t)|\leq \begin{cases}
\frac{C_0 }{\gamma_0-\kappa_0}t^{-\kappa_0} +\left (|\xi(t_0)|-\frac{C_0 }{\gamma_0-\kappa_0}t_0^{-\kappa_0}\right) \left(\frac{t}{t_0}\right)^{-\gamma_0}, & \gamma_0\neq \kappa_0 \\
 C_0  t^{-\kappa_0} \log t+\left (|\xi(t_0)|   - C_0 t_0^{ -\kappa_0}\log t_0  \right) \left(\frac{t}{t_0}\right)^{-\kappa_0}, & \gamma_0= \kappa_0
\end{cases}.
\end{gather*}
In this case, $\xi(t)=\mathcal O(t^{-(1-\epsilon)\varkappa_0})$ as $t\to\infty$ for all $\epsilon\in (0,1)$.
\end{proof}

Consider now an axillary system of ordinary differential equations
\begin{gather}\label{upeq1}
		 \frac{du}{dt}=\Lambda(u,t)+F(u,\psi,t),\quad
		 \frac{d\psi}{dt}=\nu(u)+G(u,\psi,t), \quad t\geq t_\ast>0,
\end{gather}
which is obtained from system \eqref{Veq} by dropping the stochastic part. Let us show that the solution $u_0(t)$ of system \eqref{ueq1} is stable with respect to the perturbation $F(u,\psi,t)$. 

\begin{Lem}
Let $(n,p,q)\in\Sigma_1$ and assumption \eqref{ass21} hold with $\lambda_{n}+\delta_{n,q}\vartheta_0<0$. Then for all $\varepsilon>0$ there exist $\delta_0>0$ and $t_0>0$ such that any solution $(u(t),\psi(t))$ of system \eqref{upeq1} with initial data $|u(t_0)-u_0(t_0)|<\delta_0$, $\psi(t_0)\in[0,2\pi)$ satisfies $|u(t)-u_0(t)|<\varepsilon$ as $t\geq t_0$. Moreover, if $n=q$, then
\begin{gather}\label{asup}
	u(t)=t^{-\vartheta_0}  \left(\xi_0+\mathcal O(t^{-(1-\epsilon)\varkappa_0})\right), \quad 
	\frac{\psi(t)}{t}=1+\mathcal O\left(t^{-\varkappa_0}\right)+\mathcal O(t^{-1}\log t), \quad t\to\infty
\end{gather}
for all $\epsilon\in (0,1)$.
\end{Lem}
\begin{proof}
Substituting 
\begin{gather}\label{subs1}
u(t)=u_0(t)+t^{-\vartheta_0}\eta(t)
\end{gather}
into the first equation of \eqref{upeq1} yields
\begin{gather}\label{zeq1}
	\frac{d\eta}{dt}= \mathcal Z(\eta,t)+\mathcal F(\eta,\psi,t),
\end{gather}
where 
\begin{gather}\label{ZF}
\begin{split}
&\mathcal Z(\eta,t) \equiv t^{\vartheta_0} \left(\Lambda(u_0(t)+t^{-\vartheta_0}\eta,t)-\Lambda(u_0(t),t)\right)+\vartheta_0 t^{-1}\eta, \\ 
&\mathcal F(\eta,\psi,t)  \equiv  t^{\vartheta_0} F(u_0(t)+t^{-\vartheta_0}\eta,\psi,t).
\end{split}
\end{gather}
Note that 
\begin{gather}\label{ZFest}
	\mathcal Z(\eta,t)=t^{-\frac{n}{q}} \eta \left(\lambda_{n}+\delta_{n,q} \vartheta_0+\mathcal O(t^{-\kappa_0})\right)+\mathcal O(t^{-\frac{2p}{q}})\mathcal O(\eta^2), \quad \mathcal F(\eta,\psi,t)=\mathcal O\left(t^{-\frac{n+N-2p+1}{q}}\right)
\end{gather}
as $\eta \to 0$ and $t\to \infty$, uniformly for all $\psi\in\mathbb R$. Hence, there exist $C_0>0$, $s_0\geq t_\ast$ and $\Delta_0>0$ such that ${d|\eta|}/{dt}\leq  t^{-{n}/{q}}\left(- \gamma_0 |\eta|  +C_0  t^{-\kappa_0}\right)
$
for all $t\geq s_0$, $|\eta|\leq \Delta_0$ and $\psi\in\mathbb R$. 
Repeating the argument used in the previous lemma, it can be seen that for all $\varepsilon\in (0,\Delta_0)$ there exist $\delta_0\in (0,\varepsilon)$ and $t_0\geq s$ such that any solution of \eqref{zeq1} with initial data $|\eta(t_0)|\leq \delta_0$ cannot leave the domain $|\eta|\leq \varepsilon$ as $t\geq t_0$ for all $\psi(t)\in\mathbb R$. Combining this with \eqref{subs1}, we obtain the stability of the solution $u_0(t)$ with respect to the perturbation $F(u,\psi,t)$. Moreover, as in the proof of the previous lemma, we can show that if $n=q$, then the solution $\eta(t)$ of equation \eqref{zeq1} with initial data $|\eta(t_0)|\leq \delta_0$ has the asymptotics $\eta(t)=\mathcal O(t^{-(1-\epsilon)\varkappa_0})$ as $t\to\infty$ for all $\epsilon\in (0,1)$ and for all $\psi(t)\in\mathbb R$. Taking into account \eqref{subs1}, we obtain asymptotic estimate \eqref{asup} for $u(t)$ as $t\to\infty$. In this case, 
\begin{gather*}
\nu(u(t))=1+\mathcal O(t^{-\vartheta_0}), \quad 
G(u(t),\psi,t)=\mathcal O(t^{-\frac{1}{q}})+\mathcal O(t^{-\frac{n}{q}}), \quad t\to\infty  
\end{gather*}
uniformly for all $\psi\in\mathbb R$. Combining this with the second equation of system \eqref{upeq1}, we get the asymptotic behaviour of $\psi(t)/t$ as $t\to\infty$. 
\end{proof}

Thus, the assumption \eqref{ass21} with  $\lambda_{n}+\delta_{n,q}\vartheta_0<0$ guarantees that $u(t)\to 0$ as $t\to\infty$ for solutions of the reduced system \eqref{upeq1}. Let us show that the full stochastic system \eqref{Veq} has a similar property. 

\begin{proof}[Proof of Theorem~\ref{Th2}]
Fix the parameters $\varepsilon_1>0$, $\varepsilon_2>0$, and consider the auxiliary function 
\begin{gather*}
	\Theta_0(x,y,t)\equiv t^{\vartheta_0} \left(V_N (I(x,y),\Phi(x,y),t)-u_0(t)\right)
\end{gather*}
with $N=2p+q-n$, where $u_0(t)$ is the solution of \eqref{ueq1} with asymptotics \eqref{uastas}. From \eqref{exch11} and \eqref{est}  it follows that there exists $M_1>0$ such that
\begin{gather}\label{axest}
|\rho ({\bf z},t;\zeta_0(t),\vartheta_0)|-M_1  t^{-\frac{1}{q}} \leq |\Theta_0(x,y,t)|\leq |\rho ({\bf z},t;\zeta_0(t),\vartheta_0)|+M_1  t^{-\frac{1}{q}}
\end{gather}
for all $ \mathfrak D_0(d_0,s_0)$, where
\begin{gather*}
	\mathfrak D_0(d_0,s_0)=\{(x,y,t)\in \mathcal D(0,E_0)\times\{t\geq s_0\}: \ \  |\rho({\bf z},t;\zeta_0(t),\vartheta_0)|\leq d_0\}
\end{gather*}
with some $d_0>0$ and $s_0\geq \max \{t_\ast, ((d_0+\max_{t\geq t_\ast} \zeta_0(t))/E_0)^{1/\vartheta_0}\}$. It can easily be checked that 
\begin{align*}
&L\Theta_0(x,y,t)\equiv  \vartheta_0 t^{-1} \Theta_0(x,y,t)-t^{\vartheta_0} u_0'(t)+t^{\vartheta_0} L V_N (I(x,y),\Phi(x,y),t), \\ 
&LV_N (I(x,y),\Phi(x,y),t) \equiv \mathcal L V_N(E,\varphi,t)\big|_{E=I(x,y),\varphi=\Phi(x,y)}.
\end{align*}
Combining this with \eqref{ZF}, we get
\begin{gather*} 
  L\Theta_0(x,y,t)\equiv \mathcal Z(\Theta_0(x,y,t),t)+\mathcal F(\Theta_0(x,y,t),\Phi(x,y),t).
\end{gather*}
It follows from \eqref{ZFest} that for all $\epsilon\in (0,1)$ there exist $0<d_1\leq d_0$ and $s_1\geq s_0 $ such that
\begin{align*}
	 & {\hbox{\rm sgn}}\left( \Theta_0(x,y,t) \right)\mathcal Z(\Theta_0(x,y,t),t)\leq -t^{-\frac{n}{q}}(1-\epsilon)\gamma_0\left|\Theta_0(x,y,t)\right|, \\
	&	\left|\mathcal F\left(\Theta_0(x,y,t),\Phi(x,y),t\right)\right|\leq M_2 t^{-\frac{n+N-2p+1}{q}}\quad 
\end{align*}
for all $(x,y,t)\in\mathfrak D_0(d_1,s_1)$ with $\gamma_0=|\lambda_n+\delta_{n,q}\vartheta_0|>0$ and $M_2={\hbox{\rm const}}>0$.
Hence,
\begin{gather}\label{Luest}
	L|\Theta_0(x,y,t)|\leq  t^{-\frac{n}{q}}\left(- (1-\epsilon)\gamma_0\left|\Theta_0(x,y,t)\right| + M_2 t^{- \frac{N-2p+1}{q} }\right).
\end{gather} 
 
Consider
\begin{gather*}
	\mathcal V_0(x,y,t)=|\Theta_0(x,y,t)| +(M_1 +  q M_2)  t^{-\frac{1}{q}} 
\end{gather*}
 as a stochastic Lyapunov function candidate for system \eqref{FulSys}. Taking into account \eqref{axest} and \eqref{Luest}, we get
\begin{gather}\label{LU3est}
	\begin{split}
	&|\rho ({\bf z},t;\zeta_0(t),\vartheta_0)|\leq \mathcal V_0(x,y,t)\leq |\rho ({\bf z},t;\zeta_0(t),\vartheta_0)|+ M t^{-\frac{1}{q}},\\
	& L\mathcal V_0(x,y,t)\leq - \left(t^{-\frac{n}{q}} (1-\epsilon)\gamma_0 \left|\Theta_0(x,y,t)\right|+   \frac{M_1}{q} t^{-\frac{1}{q}-1}\right)\leq 0
	\end{split}
\end{gather}
 for all $(x,y,t)\in \mathfrak D_0(d_1,s_1)$ with $M=2M_1+ q M_2>0$.

Let ${\bf z}(t)\equiv (x(t),y(t))^T$ be a solution of system \eqref{FulSys} with initial data ${\bf z}_0=(x(t_0),y(t_0))^T$ such that $|\rho ({\bf z}_0,t_0;\zeta_0(t_0),\vartheta_0)|\leq \delta_0$ and $\tau_{\mathfrak D_0}$ be the first exit time of $({\bf z}(t),t)$ from the domain $ \mathfrak D_0(\varepsilon_1,t_0)$ with some $0<\delta_0<\varepsilon_1\leq d_1$ and $t_0\geq s_1$. Define the function $\theta_t=\min\{\tau_{\mathfrak D_0},t\}$, then $({\bf z}(\theta_t),\theta_t)$ is the process stopped at the first exit time from the domain $ \mathfrak D_0(\varepsilon_1,t_0)$. It follows from \eqref{LU3est} that $\mathcal V_0(x(\theta_t),y(\theta_t),\theta_t)$ is a non-negative supermartingale~\cite[\S 5.2]{RH12}. In this case, the following estimates hold:
\begin{align*}
\mathbb P\left(\sup_{t\geq t_0} |\rho ({\bf z}(t),t;\zeta_0(t),\vartheta_0)| \geq \varepsilon_1\right)
& =	\mathbb P\left(\sup_{t\geq t_0}|\rho({\bf z}(\theta_t),\theta_t;\zeta_0(\theta_t),\vartheta_0)|\geq \varepsilon_1\right)\\
& \leq \mathbb P\left(\sup_{t\geq t_0}\mathcal V_0(x(\theta_t),y(\theta_t),\theta_t)\geq \varepsilon_1\right)\leq \frac{\mathcal V_0(x(t_0),y(t_0),t_0)}{\varepsilon_1}.
\end{align*}
The last estimate follows from Doob's inequality for supermartingales. From \eqref{LU3est} it follows that $\mathcal V_0(x(t_0),y(t_0),t_0)\leq \delta_0+Mt_0^{-1/q}$. Hence, taking 
\begin{gather*}
\delta_0=\frac{\varepsilon_1\varepsilon_2}{2}, \quad  
t_0=\max\left\{s_1,\left(\frac{2M}{\varepsilon_1 \varepsilon_2}\right)^{q}\right\},
\end{gather*} 
we obtain \eqref{defineq2}. 
\end{proof}

\section{Proof of Theorem~\ref{Th3}}
\label{Sec5}

\begin{proof}[Proof of Lemma~\ref{Lem3}]
Substituting $u(t)=t^{-\vartheta_j}(\xi_j+\xi(t))$ into equation \eqref{ueq1} yields
\begin{gather*}
	\frac{d\xi}{dt}=\tilde \Lambda_j(\xi+\xi_j,t), \quad \tilde \Lambda_j(\zeta,t)\equiv t^{\vartheta_j}\Lambda\left(t^{-\vartheta_j} \zeta,t\right)+\vartheta_j t^{-1} \zeta,  
\end{gather*}
where $j\in\{1,2,3\}$ and $\vartheta_3=\vartheta_2$. Define the parameters
\begin{gather*}
\chi_1=\frac{n+d}{q},\quad
\chi_2=\frac{2p(m-1)+n}{mq},\quad
\chi_3=\chi_2. 
\end{gather*}

In {\bf Case I}, we take $j=1$. Then 
$	\tilde \Lambda_1(\zeta,t)=t^{-\chi_1} \left(Q_1(\zeta)+\mathcal O(t^{-\kappa_1})\right)$ as $t\to\infty$.
Since $Q_1(\xi_1)=0$ and $Q_1'(\xi_1)=-(m-1)|\lambda_{n+d}+\delta_{n+d,q} \vartheta_1|<0$, we obtain 
\begin{gather*}
\tilde \Lambda_1(\xi+\xi_1,t)= t^{-\chi_1}\left(-|Q_1'(\xi_1)|\xi+\mathcal O(\xi^2)+\mathcal O(t^{-\kappa_1})\right), \quad t\to\infty,  \quad \xi\to 0.
\end{gather*}

In {\bf Case II}, we take $j=2$ and get 
$\tilde \Lambda_2(\zeta,t)=t^{-\chi_2} \left(Q_2(\zeta)+\mathcal O(t^{-\kappa_2})\right)$ as $t\to\infty$.
Since $Q_2(\xi_2)=0$ and $Q_2'(\xi_2)=-|\lambda_{n,m} m \xi_2^{m-1}|<0$, it follows that 
\begin{gather*}
\tilde \Lambda_2(\xi+\xi_2,t)= t^{-\chi_2}\left(-|Q_2'(\xi_2)|\xi+\mathcal O(\xi^2)+\mathcal O(t^{-\kappa_2})\right), \quad t\to\infty,  \quad  \xi\to 0.
\end{gather*}

In {\bf Case III}, we take $j=3$ and obtain 
$
	\tilde \Lambda_3(\zeta,t)=t^{-\chi_3} \left(Q_3(\zeta)+\mathcal O(t^{-\kappa_3})\right)
$
as $t\to\infty$.
It can easily be checked that there exists $\xi_3>0$ such that $Q_3(\xi_3)=0$ and $Q_3'(\xi_3)<0$. Hence,
\begin{gather*}
\tilde \Lambda_3(\xi+\xi_3,t)= t^{-\chi_3}\left(-|Q_3'(\xi_3)|\xi+\mathcal O(\xi^2)+\mathcal O(t^{-\kappa_2})\right), \quad t\to\infty,  \quad  \xi \to 0.
\end{gather*}

Thus, in all three cases, for all $\epsilon\in(0,1)$ there exist $C_0>0$, $s_0\geq t_\ast$ and $\Delta_0>0$ such that 
\begin{gather}\label{xiests}
\frac{d|\xi|}{dt}\leq t^{-\chi_j}\left(-\gamma_j |\xi|+C_0 t^{- \kappa_j}\right)
\end{gather}
for all $t\geq s_0$ and $|\xi|\leq \Delta_0$, where $\gamma_j=(1-\epsilon)|Q_j'(\xi_j)|>0$ and $\kappa_3=\kappa_2$.
Note that in {\bf Case II}, $\chi_2<(n+d)/q\leq 1$. 

For all $\varepsilon\in (0,\Delta_0)$ define
\begin{gather*}
\delta_j=\frac{2 C_0 s_0^{-\kappa_j}}{ \gamma_j}<\varepsilon, \quad t_0=\max\left\{s_0, \left(\frac{4 C_0}{ \varepsilon \gamma_j }\right)^{\frac{1}{\kappa_j}}\right\}.
\end{gather*}
Then $d|\xi(t)|/dt<0$ for solutions of \eqref{zeq1} such that $\delta_j \leq |\xi(t)|\leq \varepsilon$ as $t\geq t_0$. Hence, any solution of \eqref{xieq1} with initial data $|\xi(t_0)|\leq \delta_j$ cannot leave the domain $|\xi|\leq \varepsilon$ as $t\geq t_0$.

By integrating \eqref{xiests} with respect to $t$, we get
\begin{eqnarray*}
  |\xi(t)|&\leq& |\xi(s_0)| \exp\left\{- \gamma_j (\tau_j(t)-\tau_j(s_0))\right\}\\
					& &	+C_0 \int\limits_{s_0}^t 
	\exp\left\{ \gamma_j (\tau_j(\varsigma)-\tau_j(t))\right\} \varsigma^{-\chi_j- \kappa_j}\, d\varsigma, \quad j\in\{1,2,3\},
\end{eqnarray*}
where $(\tau_j(t))'=t^{-\chi_j}$. Using the same type of reasoning as in the proof of Lemma~\ref{Lem1}, we obtain asymptotic estimates \eqref{u1as}, \eqref{u2as}  and \eqref{u3as}.
\end{proof}

\begin{proof}[Proof of Theorem~\ref{Th3}]
Consider the functions 
\begin{gather*}
	\Theta_j(x,y,t)\equiv t^{\vartheta_j} \left(V_N (I(x,y),\Phi(x,y),t)-u_j(t)\right), \quad j\in\{1,2,3\}
\end{gather*}
with $N\geq q(1+\vartheta_j)$, where $\vartheta_3=\vartheta_2$, and $u_1(t)$, $u_2(t)$ and $u_3(t)$ are the solutions of equation \eqref{ueq1} with asymptotics \eqref{u1as}, \eqref{u2as} and \eqref{u3as}, respectively. It can easily be checked that
\begin{eqnarray*}
L\Theta_j(x,y,t)&\equiv&  \vartheta_j t^{-1} \Theta_j(x,y,t)-t^{\vartheta_j} u_j'(t)+t^{\vartheta_j} L V_N (I(x,y),\Phi(x,y),t),\\
LV_N (I(x,y),\Phi(x,y),t) & \equiv &\mathcal L V_N(E,\varphi,t)\big|_{E=I(x,y),\varphi=\Phi(x,y)}\\
&\equiv& \mathcal Z_j(\Theta_j(x,y,t),t)+\mathcal F_j(\Theta_j(x,y,t),\Phi(x,y),t),
\end{eqnarray*}
where
\begin{align*} 
&\mathcal Z_j(\eta,t) \equiv t^{\vartheta_j} \left(\Lambda(u_j(t)+t^{-\vartheta_j}\eta,t)-\Lambda(u_j(t),t)\right)+\vartheta_j t^{-1}\eta =t^{-\chi_j} \eta \left(Q_j'(\xi_j)+\mathcal O(t^{-\varkappa_j})+ \mathcal O(\eta)\right),\\
& \mathcal F_j(\eta,\psi,t)  \equiv  t^{\vartheta_j} F(u_j(t)+t^{-\vartheta_j}\eta,\psi,t)=\mathcal O\left(t^{\vartheta_j-\frac{N+1}{q}}\right)
\end{align*}
as $\eta \to 0$ and $t\to \infty$ uniformly for all $\psi\in\mathbb R$. Note that $\vartheta_j-(N+1)/q\leq -1 -1/q$.
From \eqref{exch11} and \eqref{est} it follows that there exists $M_1>0$ such that 
\begin{gather}\label{axestj}
|\rho({\bf z},t;\zeta_j(t),\vartheta_j)|-M_1  t^{-\frac{1}{q}} \leq |\Theta_j(x,y,t)|\leq |\rho ({\bf z},t;\zeta_j(t),\vartheta_j)|+M_1  t^{-\frac{1}{q}}
\end{gather}
for all $(x,y,t)\in \mathfrak D_j(d_0,s_0)=\{(x,y,t)\in \mathcal D(0,E_0)\times\{t\geq s_0\}: |\rho({\bf z},t;\zeta_j(t),\vartheta_j)|\leq d_0\}$ with some $d_0>0$ and $s_0\geq \max \{t_\ast, ((\max_{t\geq t_\ast}\zeta_j(t)+d_0)/E_0)^{1/\vartheta_j}\}$. 
Moreover, for all $\epsilon\in (0,1)$ there exist $0<d_1\leq d_0$ and $s_1\geq s_0 $ such that
\begin{gather}\begin{split}\label{estj}
	 & {\hbox{\rm sgn}}\left( \Theta_j(x,y,t) \right)\mathcal Z_j(\Theta_j(x,y,t),t)\leq -t^{-\chi_j}(1-\epsilon)|Q'_j(\xi_j)|\left|\Theta_j(x,y,t)\right|, \\
	&	\left|\mathcal F_j\left(\Theta_j(x,y,t),\Phi(x,y),t\right)\right|\leq M_2 t^{-1-\frac{1}{q}}\quad 
\end{split}
\end{gather}
for all $(x,y,t)\in\mathfrak D_j(d_1,s_1)$ with $M_2={\hbox{\rm const}}>0$.

In each of three cases, the Lyapunov function candidate for system \eqref{FulSys} can be constructed in the following form:
\begin{gather*}
	\mathcal V_j(x,y,t)=|\Theta_j(x,y,t)| +(M_1+  q M_2)  t^{-\frac{1}{q}} 
\end{gather*}
with the corresponding value of $j\in \{1,2,3\}$. From  \eqref{axestj} and \eqref{estj} it follows that
\begin{gather*}
	\begin{split}
	&|\rho ({\bf z},t;\zeta_j(t),\vartheta_j)|\leq \mathcal V_j(x,y,t)\leq |\rho ({\bf z},t;\zeta_j(t),\vartheta_j)|+ M t^{-\frac{1}{q}},\\
	& L\mathcal V_j(x,y,t)\leq - \left(t^{-\chi_j} (1-\epsilon)|Q_j'(\xi_j)|\left|\Theta_j(x,y,t)\right|+ \frac{M_1}{q} t^{-\frac{1}{q}-1}\right)\leq 0
	\end{split}
\end{gather*}
 for all $(x,y,t)\in \mathfrak D_j(d_1,s_1)$ with $M=2M_1+ q M_2>0$. Note that the last estimates are similar to \eqref{LU3est}. Hence, the rest of the proof is the same as that of Theorem~\ref{Th2} with $\mathcal V_j(x,y,t)$ and $\rho ({\bf z},t;\zeta_j(t),\vartheta_j)$ instead of $\mathcal V_0(x,y,t)$ and $\rho ({\bf z},t;\zeta_0(t),\vartheta_0)$, respectively.
\end{proof}

\section{Proof of Theorem~\ref{Th4}}
\label{Sec6}

\begin{proof}[Proof of Lemma~\ref{Lem4}] Substituting $u(t)= \xi_\ast+ \xi(t) $ into equation \eqref{ueq1} yields
\begin{gather*}
	\frac{d\xi}{dt}=\Lambda(\xi_\ast+ \xi,t)=t^{-\frac{n}{q}}\left(\Lambda_n'(\xi_\ast)\xi+\mathcal O(\xi^2)+ \mathcal O(t^{-\frac{1}{q}})\right)
\end{gather*}
as $t\to \infty$ and $\xi\to 0$. Hence, for all $\epsilon\in (0,1)$ there exist $C_0>0$, $s_0\geq t_\ast$ and $\Delta_0>0$ such that 
\begin{gather*}
\frac{d|\xi|}{dt}\leq  t^{-\frac{n}{q}} \left(-\gamma_\ast |\xi|+C_0 t^{-\frac{1}{q}}\right)
\end{gather*}
for all $t\geq s_0$ and $|\xi|\leq \Delta_0$, where $\gamma_\ast=(1-\epsilon)|\Lambda_{n}'(\xi_\ast)|>0$. This inequality is similar to \eqref{xiineqlem1}. Therefore, using the same type of reasoning as in the proof of Lemma~\ref{Lem1} with $1/q$, $\gamma_\ast$ and $\varkappa_\ast$ instead of $\kappa_0$, $\gamma_0$ and $\varkappa_0$, respectively, we obtain that for all $\varepsilon>0$ there exist $\delta_0>0$ and $t_0\geq s_0$ such that any solution with initial data $|\xi(t_0)|\leq \delta_0$ cannot leave the domain $|\xi|\leq \varepsilon$ as $t\geq t_0$. Similarly, we get asymptotic estimate \eqref{uastasc}.
\end{proof}

\begin{proof}[Proof of Theorem~\ref{Th4}]
Let $u_\ast(t)$ be the solution of \eqref{ueq1} with asymptotics \eqref{uastasc}. Define
\begin{gather*}
	\Theta_\ast(x,y,t)\equiv  \left(V_N (I(x,y),\Phi(x,y),t)-u_\ast(t)\right)^2
\end{gather*}
with $N\geq q$. It follows from \eqref{exch11} and \eqref{est} that there exists $M_1>0$ such that
\begin{gather}\label{axestc}
|\Theta_\ast(x,y,t)-\rho^2 ({\bf z},t;u_\ast(t),0)|\leq M_1  t^{-\frac{1}{q}}
\end{gather}
for all $(x,y,t)\in \mathfrak D_\ast(d_0,s_0)=\{(x,y,t)\in \mathcal D(0,E_0)\times\{t\geq s_0\}: |\rho({\bf z},t;u_\ast(t),0)|\leq d_0\}$ with some $d_0>0$ and $s_0\geq t_\ast$. 
Note that 
\begin{align*}
 L\Theta_\ast(x,y,t)&\equiv {\hbox{\rm tr}}({\bf A}^T {\bf M} {\bf A})+
 2\left(V_N (I(x,y),\Phi(x,y),t)-u_\ast(t)\right)\left(\mathcal L V_N(E,\varphi,t)|_{E=I(x,y),\varphi=\Phi(x,y)}-u_\ast'(t)\right)\\
&  ={\hbox{\rm tr}}({\bf A}^T {\bf M} {\bf A})+ 2 \Theta_\ast(x,y,t) t^{-\frac{2p}{q}}  \left[\Lambda_{2p}'(u_\ast(t))+\mathcal O(t^{-\frac{1}{q}})+ \mathcal O (\rho  )\right]+\mathcal O\left(t^{-\frac{N+1}{q}}\right)
\end{align*}
as $\rho=\rho ({\bf z},t;u_\ast(t),0)\to 0$ and $t\to \infty$, where
\begin{gather*}
 {\bf M}({\bf z},t)  \equiv  \left.
\begin{pmatrix}
(\partial_xV_N)^2 &  \partial_xV_N \partial_yV_N \\
 \partial_xV_N \partial_yV_N  & (\partial_yV_N)^2
\end{pmatrix}\right|_{E=I(x,y), \varphi=\Phi(x,y)}.
\end{gather*}
It follows from \eqref{H0as}, \eqref{asslimc} and \eqref{est} that there exists $M_0>0$ such that
\begin{gather*}
\big|{\hbox{\rm tr}} ({\bf A}^T {\bf M} {\bf A})\big |\leq  \mu M_0 t^{-\frac{2p}{q}}
\end{gather*}
for all $(x,y,t)\in \mathfrak D_\ast(d_0,s_0)$. From the proof of Lemma~\ref{Lem4} it follows that for all $\epsilon>0$ there exists $s_1\geq s_0$ such that $\Lambda_{2p}'(u_\ast(t))\leq -\gamma_\ast$ as $t\geq s_1$ with $\gamma_\ast=(1-\epsilon)|\Lambda'_{2p}(\xi_\ast)|>0$. Therefore, there exist $0<d_1\leq d_0$ and $s_2\geq s_1 $ such that 
\begin{gather}\label{estc}
 L\Theta_\ast(x,y,t)
	\leq   t^{-\frac{2p}{q}}\left(-\gamma_\ast  \Theta_\ast(x,y,t) + \mu M_0 \right)+M_2 t^{-\frac{N+1}{q}}
\end{gather}
for all $(x,y,t)\in\mathfrak D_\ast(d_1,s_2)$ and with some $M_2={\hbox{\rm const}}>0$.

We take $N=q$ and consider the Lyapunov function candidate for system \eqref{FulSys} in the following form:
\begin{gather*}
	\mathcal V_\ast(x,y,t)= \Theta_\ast(x,y,t)+\mu M_0\zeta(t)+(M_1+ q M_2)  t^{-\frac{1}{q}}
\end{gather*}
with 
\begin{gather*}
\zeta(t)\equiv 
\begin{cases}
\displaystyle t_0^{-\frac{2p}{q}} (\mathcal T+t_0-t), & 2p<q,\\
\displaystyle \log\left(\frac{\mathcal T+t_0}{t}\right),& 2p=q,\\
\displaystyle \int\limits_t^{t_0+\mathcal T} \sigma^{-\frac{2p}{q}}\,d\sigma, & 2p>q
\end{cases}
\end{gather*}
and some $t_0\geq s_2$.
From  \eqref{axestc} and \eqref{estc} it follows that
\begin{gather}\label{Vastest}
	  \mathcal V_\ast(x,y,t)\geq \rho^2 ({\bf z},t;u_\ast(t),0),\quad
	  L\mathcal V_\ast(x,y,t)\leq -t^{-\frac{2p}{q}}\gamma_\ast \Theta_\ast(x,y,t)  -\frac{M_1}{q}t^{-1-\frac{1}{q}} \leq 0
\end{gather}
 for all $(x,y,t)\in \mathfrak B(d_1,t_0,\mathcal T)=\{(x,y,t)\in \mathfrak D_\ast(d_1,t_0), t_0\leq t\leq t_0+\mathcal T\}$. 

Fix the parameters $\varepsilon_1>0$ and $\varepsilon_2>0$. Let ${\bf z}(t)\equiv (x(t),y(t))^T$ be a solution of system \eqref{FulSys} with initial data ${\bf z}_0=(x(t_0),y(t_0))^T$ such that $|\rho ({\bf z}_0,t_0;u_\ast(t_0),0)|\leq \delta_0$ and $\tau_{\mathfrak B}$ be the first exit time of $({\bf z}(t),t)$ from the domain $ \mathfrak B(\varepsilon_1,t_0,\mathcal T)$ with some $0<\delta_0<\varepsilon_1\leq d_1$. Define the function $\theta_t=\min\{\tau_{\mathfrak B},t\}$, then $({\bf z}(\theta_t),\theta_t)$ is the process stopped at the first exit time from the domain $ \mathfrak B(\varepsilon_1,t_0,\mathcal T)$. It follows from \eqref{Vastest} that $\mathcal V_\ast(x(\theta_t),y(\theta_t),\theta_t)$ is a non-negative supermartingale, and the following estimates hold:
\begin{align*}
\mathbb P\left(\sup_{t_0\leq t\leq t_0+\mathcal T} |\rho ({\bf z}(t),t;u_\ast(t),0) | \geq \varepsilon_1\right)
& =	\mathbb P\left(\sup_{t\geq t_0}\rho^2({\bf z}(\theta_t),\theta_t;u_\ast(\theta_t),0)\geq \varepsilon_1^2\right)\\
& \leq \mathbb P\left(\sup_{t\geq t_0}\mathcal V_\ast(x(\theta_t),y(\theta_t),\theta_t)\geq \varepsilon_1^2\right)\leq \frac{\mathcal V_\ast(x(t_0),y(t_0),t_0)}{\varepsilon_1^2}.
\end{align*}
Since $\mathcal V_\ast(x(t_0),y(t_0),t_0)\leq \delta_0^2+Mt_0^{-1/q}+\mu M_0 \zeta(t_0)$  with $M=2M_1+ q M_2>0$. Hence, taking  $\delta_0={\varepsilon_1^2\varepsilon_2}/{3}$ and
\begin{gather*}
\begin{cases}
\mathcal T=\frac{\delta_0^{2}}{ \mu M_0}  t_0^{\frac{2p}{q}}, \quad t_0=\max\left\{s_2,\left(\frac{M}{\delta_0^{2}}\right)^2\right\},& 2p<q,\\
\mathcal T=t_0\left(e^{\frac{\delta_0^{2}}{\mu M_0}}-1\right), \quad t_0=\max\left\{s_2,\left(\frac{M}{\delta_0^{2}}\right)^2\right\}, & 2p=q,\\
\mathcal T=\infty, \quad t_0=\max\left\{s_2,\left(\frac{M}{\delta_0^{2}}\right)^2, \left(\frac{\mu M_0 q}{(2p-q)\delta_0^2}\right)^{\frac{q}{2p-q}}\right\}, & 2p>q,
\end{cases}
\end{gather*} 
we obtain \eqref{defineq6}.
\end{proof}

\section{Proof of Theorem~\ref{Th5}} \label{Sec7}
\begin{proof}[Proof of Lemma~\ref{Lem5}] Note that $1-n/q\leq -1/q$ for $(n,p,q)\in\Sigma_3$. By substituting 
\begin{gather*}
u(t)= \xi_\ast+ t^{1-\frac{n}{q}}\frac{q\Lambda_n(\xi_\ast)}{q-n} +\xi(t)
\end{gather*}
into equation \eqref{ueq1}, we obtain
\begin{gather*}
	\frac{d\xi}{dt}=\Lambda\left(\xi_\ast+t^{1-\frac{n}{q}}\frac{q\Lambda_n(\xi_\ast)}{q-n} +\xi,t\right)-t^{-\frac{n}{q}} \Lambda_n(\xi_\ast) =t^{-\frac{n}{q}}\left(\Lambda_{n}'(\xi_\ast)\xi+\mathcal O(\xi^2)+ \mathcal O(t^{-\frac{1}{q}})\right)
\end{gather*}
as $t\to \infty$ and $\xi\to 0$. Hence, for all $\epsilon\in (0,1)$ there exist $C_0>0$, $s_0\geq t_\ast$ and $\Delta_0>0$ such that 
\begin{gather*}
\frac{d|\xi|}{dt}\leq  t^{-\frac{n}{q}} \left(-\gamma_\ast |\xi|+C_0 t^{-\frac{1}{q}}\right)
\end{gather*}
for all $t\geq s_0$ and $|\xi|\leq \Delta_0$ with $\gamma_\ast=(1-\epsilon)|\Lambda_{n}'(\xi_\ast)|>0$. Thus,  as in the proof of Lemma~\ref{Lem1}, we see that there exist $t_0\geq s_0$ and $|\xi(t_0)|\leq \Delta_0$ such that $\xi(t)=\mathcal O(t^{-1/q})$ as $ t\to\infty$.
\end{proof}

\begin{proof}[Proof of Theorem~\ref{Th5}]
Consider the auxiliary function
\begin{gather*}
	\Theta_\ast(x,y,t)\equiv  \left(V_N (I(x,y),\Phi(x,y),t)-v_\ast(t)\right)^2
\end{gather*}
with $N\geq n>q$, where $v_\ast(t)$ is the solution of \eqref{ueq1} with asymptotics \eqref{uasast}. From \eqref{exch11} and \eqref{est} it follows that there exist $M_1>0$ such that 
\begin{gather}\label{axestast}
|\Theta_\ast(x,y,t)-\rho^2 ({\bf z},t;v_\ast(t),0)|\leq M_1  t^{-\frac{1}{q}}
\end{gather}
for all $(x,y,t)\in \mathfrak D_\ast(d_0,s_0)=\{(x,y,t)\in \mathcal D(0,E_0)\times\{t\geq s_0\}: |\rho({\bf z},t;v_\ast(t),0)|\leq d_0\}$ with some $d_0>0$ and $s_0\geq t_\ast$. 
As in the proof of Theorem~\ref{Th4} it can shown that for all $\epsilon>0$ there exist $s_1\geq s_0$ and $0<d_1\leq d_0$ such that
\begin{gather}\label{estast}
	L  \Theta_\ast(x,y,t) \leq   t^{-\frac{n}{q}}\left(-\gamma_\ast \Theta_\ast(x,y,t) + \mu M_0 +M_2 t^{-\frac{N-n+1}{q}}\right)
\end{gather}
for all $(x,y,t)\in\mathfrak D_\ast(d_1,s_1)$ with $\gamma_\ast=(1-\epsilon)|\Lambda'_{n}(\xi_\ast)|>0$ and some $M_0, M_2={\hbox{\rm const}}>0$.

In this case, the Lyapunov function candidate for system \eqref{FulSys} can be considered in the form
\begin{gather*}
	\mathcal V_\ast(x,y,t)= \Theta_\ast(x,y,t) - \frac{q \mu M_0}{q-n} t^{1-\frac{n}{q}}+(M_1 +q M_2)t^{-\frac{1}{q}}  
\end{gather*}
with $N=n$. From  \eqref{axestast} and \eqref{estast} it follows that
\begin{align*}
&  \rho^2 ({\bf z},t;v_\ast(t),0) \leq \mathcal V_\ast(x,y,t)\leq  \rho^2 ({\bf z},t;v_\ast(t),0) + M t^{-\frac{1}{q}},\\ 
& L\mathcal V_\ast(x,y,t)\leq -t^{-\frac{n}{q}} \gamma_\ast \Theta_\ast(x,y,t) - \frac{ M_1}{q} t^{-1-\frac{1}{q}} \leq 0
\end{align*}
for all $(x,y,t)\in \mathfrak D_\ast(d_1,s_1)$ with $M=q\mu M_0/(n-q)+M_1 +q M_2$. We see that the constructed function $\mathcal V_\ast(x,y,t)$ satisfies the inequalities similar to \eqref{LU3est}. Hence, the rest of the proof can repeat the line of reasoning of Theorem~\ref{Th2} with $\mathcal V_\ast(x,y,t)$ and $\rho^2 ({\bf z},t;v_\ast(t),0)$ instead of $\mathcal V_0(x,y,t)$ and $\rho ({\bf z},t;\zeta_0(t),\vartheta_0)$, respectively.
\end{proof}

\section{Examples}\label{SecEx}
In this section, we consider the examples of asymptotically autonomous systems with damped stochastic perturbations and discuss the application of the proposed theory. 

\subsection{Example 1.} 
First, consider a stochastic system in the following form:
\begin{gather}\label{Ex1}
\begin{split}
&dx=y\,dt, \\ 
&dy=(-\sin x+t^{-\frac{h}{q}} a  y )\,dt+t^{-\frac{p}{q}} (c + b \sin x)\,dw_2(t),
\end{split}
\end{gather}
with $p,q,h\in\mathbb Z_+$ and parameters $a,b,c\in\mathbb R$. Note that this system is of the form \eqref{FulSys} with 
\begin{gather*}
{\bf a}({\bf z},t)\equiv {\bf a}_0({\bf z})+t^{-\frac{h}{q}}{\bf a}_h({\bf z}), \quad 
{\bf A}({\bf z},t)\equiv t^{-\frac{p}{q}} {\bf A}_p({\bf z}), \\
{\bf a}_0({\bf z})\equiv \begin{pmatrix} y \\ -\sin x\end{pmatrix}, \quad 
{\bf a}_h({\bf z})\equiv \begin{pmatrix} 0 \\ a  y \end{pmatrix}, \quad 
{\bf A}_p({\bf z})\equiv \begin{pmatrix} 0 & 0 \\ 0 & c + b \sin x\end{pmatrix},
\end{gather*}
and the corresponding limiting system \eqref{LimSys} with $H_0(x,y)=1-\cos x+y^2/2$ has a stable equilibrium at the origin. In this case, the level lines $H_0(x,y)\equiv E$ as $E \in (0, 2)$, lying in the neighbourhood of the equilibrium, correspond to $T(E)$-periodic solutions of system \eqref{LimSys} such that $\nu(E)\equiv2\pi/T(E)=1-E/8+\mathcal O(E^2)$ as $E\to 0$.

Let us remark that system \eqref{Ex1} with $h=q$ and $a=-1$ corresponds to the third Painlev\'{e} equation (see, for example,~\cite[Ch. 13]{IKNF06}) with a multiplicative stochastic perturbation.

{\bf 1.} Let $p=q=h=2$. Then the changes of the variables described in Section~\ref{Sec3} with $N=4$, $v_1(E,\varphi)\equiv v_3(E,\varphi)\equiv 0$,
\begin{eqnarray*}
v_2(E,\varphi) & = & - \frac{a}{\nu(E)}\left\{\int\limits_0^\varphi \left\{\big(Y(\varsigma,E)\big)^2\right\}_\varsigma\,d\varsigma\right\}_\varphi,\\
v_4(E,\varphi) & = &- \frac{1}{ \nu(E)}\left\{\int\limits_0^\varphi \left\{\frac 12 \big(c + b \sin X(\varsigma,E)\big)^2+ R_4(E,\varsigma)\right\}_\varsigma\,d\varsigma\right\}_\varphi,
\end{eqnarray*}
transform system \eqref{Ex1} into \eqref{Veq} with $\Lambda_1(v)\equiv \Lambda_3(v)\equiv 0$,
\begin{gather*}
\Lambda_2(v)=a \left\langle\big(Y(\varphi,v)\big)^2\right\rangle_\varphi, \quad 
\quad \Lambda_4(v)=  \left\langle\frac 12 \big(c + b \sin X(\varphi,v)\big)^2+ R_4(v,\varphi)\right\rangle_\varphi,
\end{gather*}
where
\begin{align*}
R_4(E,\varphi) \equiv &  -\partial_E\Lambda_2(E)v_2(E,\varphi)   +a Y (\varphi,E)  \Big( Y(\varphi,E)\partial_E+\partial_y \Phi(X(\varphi,E) ,Y(\varphi,E) )\partial_\varphi \Big) v_2(E,\varphi).
\end{align*}
Recall that $\langle C(\varphi)\rangle_\varphi$ is the average of a periodic function $C(\varphi)$ and $\{C(\varphi)\}_\varphi=C(\varphi)-\langle C(\varphi)\rangle_\varphi$. Note that
\begin{gather*}
\Lambda_2(v)=a v (1+\mathcal O(v)), \quad \Lambda_4(v)= \frac{c^2}{2}+\mathcal O(v), \quad v\to 0.
\end{gather*}
We see that the transformed system satisfies \eqref{ass1} and \eqref{ass21} with $n=2$, $(n,p,q)\in \Sigma_1$ and $\lambda_n=a$. In this case, $\mu_{2p}=c^2/2$, $\xi_0=c^2/(2|a+1|)$ and $\vartheta_0=1$. It follows from Theorem~\ref{Th2} that if $a<-1$, then $H_0(x(t),y(t))\approx t^{-1}\xi_0$ as $t\to\infty$ with high probability for solutions of system \eqref{Ex1} with initial data such that $|t_0 H_0(x(t_0),y(t_0))-\xi_0|$ is small enough (see Fig.~\ref{FigEx11}). 

\begin{figure}
\centering
\subfigure[$a=-2$ ]{\includegraphics[width=0.4\linewidth]{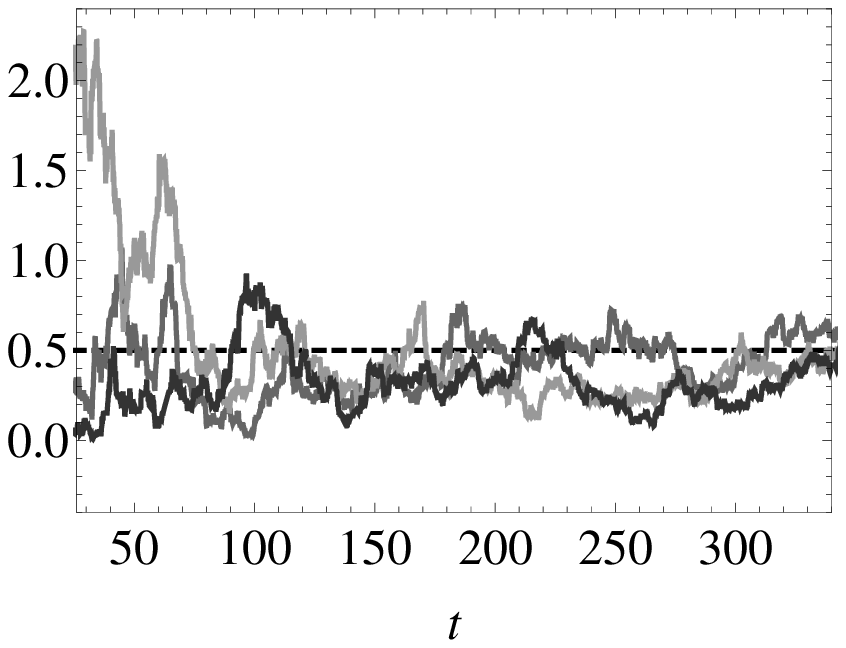}}
\hspace{4ex}
\subfigure[$a=-0.5$]{\includegraphics[width=0.4\linewidth]{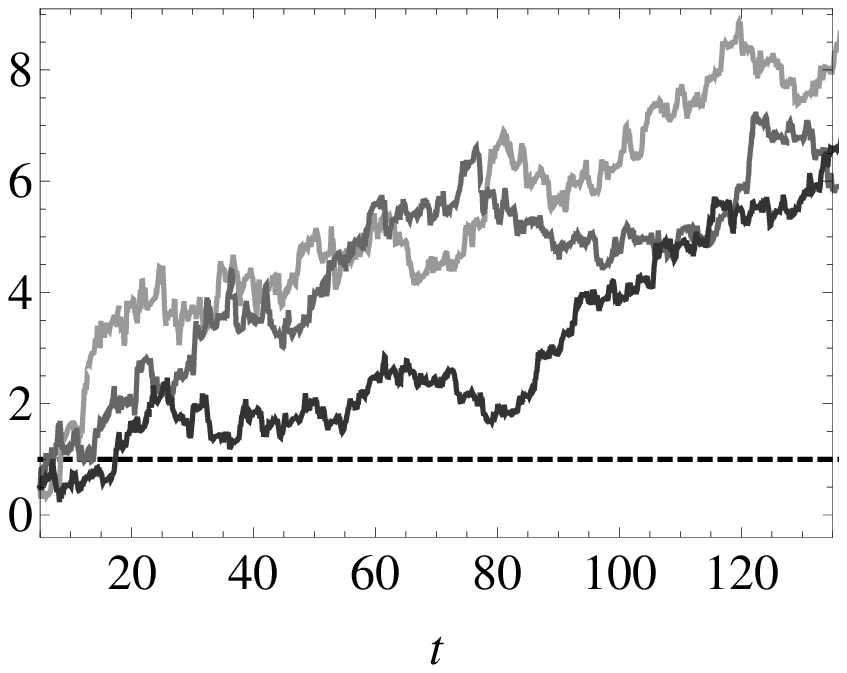}}
\caption{\small The evolution of $t H_0(x(t),y(t))$ for sample paths of the solutions to system \eqref{Ex1} with $p=q=h=2$, $b=0.5$, $c=1$ and various initial data. The dashed lines correspond to $\xi_0=c^2/(2|a+1|)$.} \label{FigEx11}
\end{figure}

{\bf 2.} Let $p=1$ and $q=h=2$. Under the changes of the variables described in Section~\ref{Sec3} with $N=2$, $v_1(E,\varphi)\equiv 0$,
\begin{eqnarray*}
v_2(E,\varphi) & = & - \frac{1}{\nu(E)}\left\{\int\limits_0^\varphi \left\{a\big(Y(\varsigma,E)\big)^2+\frac 12 \big(c + b \sin X(\varsigma,E)\big)^2\right\}_\varsigma\,d\varsigma\right\}_\varphi,
\end{eqnarray*}
system \eqref{Ex1} is transformed into \eqref{Veq} with $\Lambda_1(v)\equiv 0$ and
\begin{gather*}
\Lambda_2(v)= \left\langle a \big(Y(\varphi,v)\big)^2 +\frac 12 \big(c + b \sin X(\varphi,v)\big)^2\right\rangle_\varphi=\frac{1}{2 }\left(c^2+(2a+b^2)v +\mathcal O(v^2)\right), \quad v\to 0.
\end{gather*}
In this case, the transformed system satisfies \eqref{ass1}  with $n=2$ and $(n,p,q)\in \Sigma_2$.
If $2a+b^2<0$, then for sufficiently small $c>0$ there exists $\xi_\ast>0$ such that $\Lambda_2(\xi_\ast)=0$ and $\Lambda_2'(\xi_\ast)<0$. Hence assumption \eqref{ass3} holds with $\xi_\ast=c^2/|2a+b^2|+\mathcal O(c^3)$ and $\Lambda_2'(\xi_\ast)=-|2a+b^2|/2+\mathcal O(c^2)$ as $c\to 0$. It can easily be checked that assumption \eqref{asslimc} holds with $\mu=(|c|+|b|)^2/2$. By applying Theorem~\ref{Th4}, we find that if $2a+b^2<0$, then $H_0(x(t),y(t))\approx \xi_\ast$ on an asymptotically long time interval as $\mu\to 0$ for solutions of system \eqref{Ex1} with initial data such that $|H_0(x(t_0),y(t_0))-\xi_\ast|$ is sufficiently small (see Fig.~\ref{FigEx12}).

\begin{figure}
\centering
\subfigure[$a=-1$ ]{\includegraphics[width=0.4\linewidth]{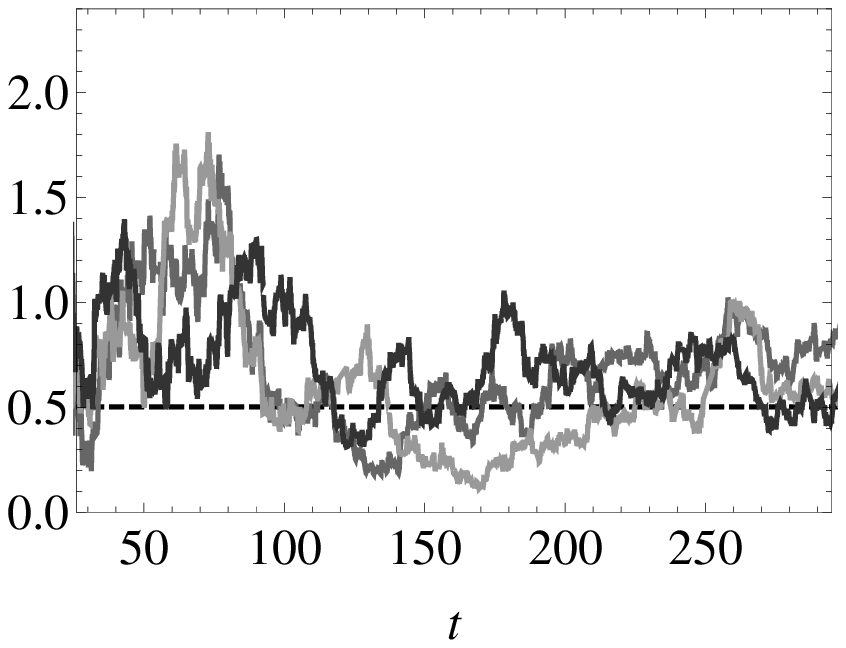}}
\hspace{4ex}
\subfigure[$a=0.5$]{\includegraphics[width=0.4\linewidth]{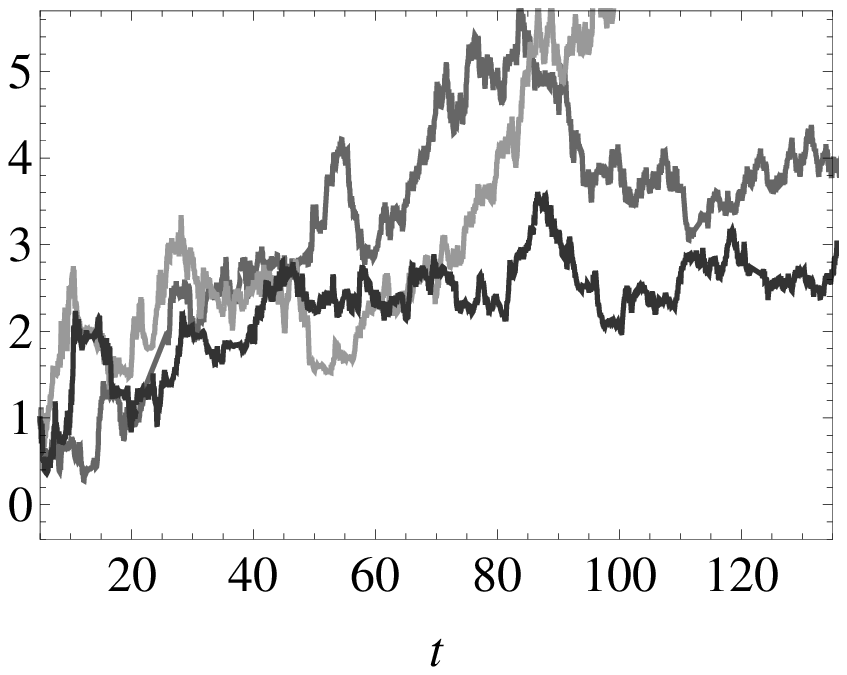}}
\caption{\small The evolution of $H_0(x(t),y(t))$ for sample paths of the solutions to system \eqref{Ex1} with $p=1$, $q=h=2$, $b=0.1$, $c=1$ and various initial data. The dashed line corresponds to $c^2/|2a+b^2|$.} \label{FigEx12}
\end{figure}

{\bf 3.} Let $p=q=2$ and $h=3$. Then the transformations described in Section~\ref{Sec3} with $N=3$, $v_1(E,\varphi)\equiv v_2(E,\varphi)\equiv0$, and
\begin{eqnarray*}
v_3(E,\varphi) & = & - \frac{a}{\nu(E)}\left\{\int\limits_0^\varphi \left\{ \big(Y(\varsigma,E)\big)^2 \right\}_\varsigma\,d\varsigma\right\}_\varphi
\end{eqnarray*}
reduce system \eqref{Ex1} to the form \eqref{Veq} with $\Lambda_1(v)\equiv\Lambda_2(v)\equiv 0$, and
\begin{gather*}
\Lambda_3(v)=a\left\langle  \big(Y(\varphi,v)\big)^2\right\rangle_\varphi=a v (1+\mathcal O(v)), \quad v\to 0.
\end{gather*}
We see that the transformed system satisfies \eqref{ass1} with $n=3$ and $(n,p,q)\in \Sigma_3$. Moreover, there exists $\varrho_0>0$ such that $\Lambda_3(v)\neq 0$ for all $v\in (0,\varrho_0)$ and $\Lambda_3'(v)=a+\mathcal O(v)$ as $v\to 0$. Hence, if $a\neq 0$, then assumption \eqref{ass4} holds with $\xi_\ast\in (0,\varrho_0)$. Applying Theorem~\ref{Th5} shows that if $a<0$, then $H_0(x(t),y(t))\approx \xi_\ast$ as $t\to\infty$ with high probability for solutions of system \eqref{Ex1} with small enough $|H_0(x(t_0),y(t_0))-\xi_\ast|$ (see Fig.~\ref{FigEx13}).

\begin{figure}
\centering
\subfigure[$a=-1$ ]{\includegraphics[width=0.4\linewidth]{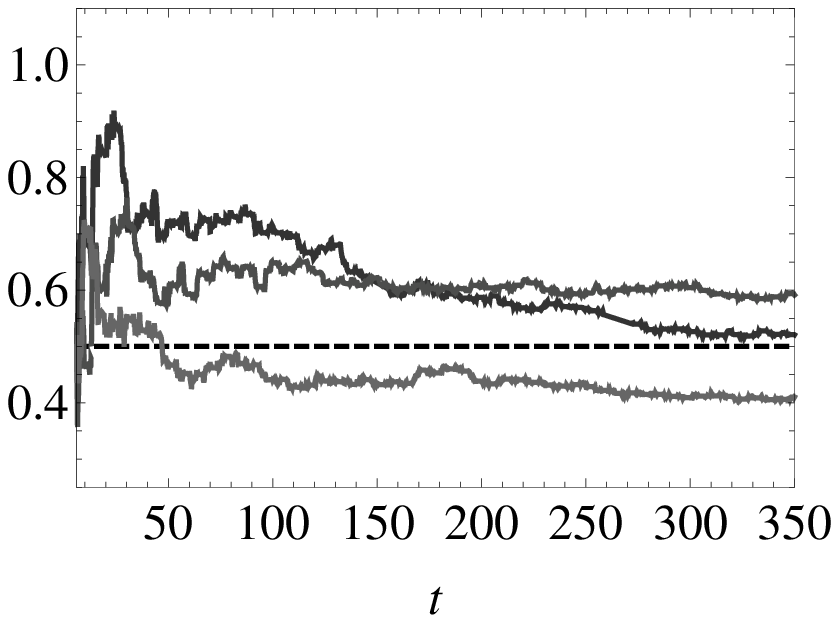}}
\hspace{4ex}
\subfigure[$a=1$]{\includegraphics[width=0.4\linewidth]{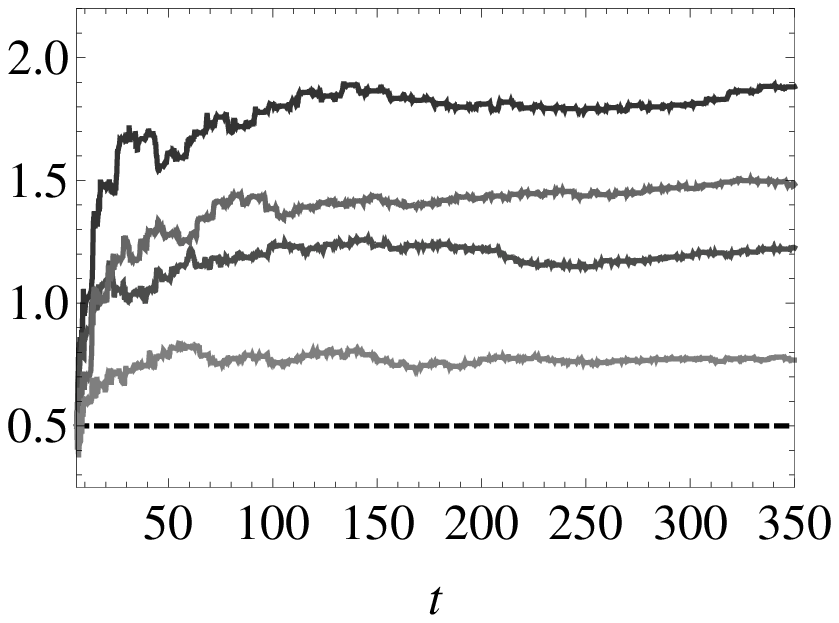}}
\caption{\small The evolution of $t H_0(x(t),y(t))$ for sample paths of the solutions to system \eqref{Ex1} with $p=q=2$, $h=3$, $b=1$, $c=\sqrt{0.5}$ and initial data $x(t_0)=0$, $y(t_0)=1$. The dashed lines correspond to $\xi_\ast=1/2$.} \label{FigEx13}
\end{figure}

\subsection{Example 2.} 
The perturbed system 
\begin{gather}\label{Ex2}
\begin{split}
&dx=y\,dt, \\ 
&dy=\left(-x+t^{-\frac{h}{q}} F_{1}(x,y)+t^{-\frac{h+d}{q}} F_{2}(y) \right)\,dt+t^{-\frac{p}{q}} c\,dw_2(t),
\end{split}\\
\nonumber F_1(x,y)\equiv \frac{a x^2y}{1+x^2}, \quad  F_2(y)\equiv b y, \quad p,q,h,d\in\mathbb Z_+, \quad a,b,c={\hbox{\rm const}}
\end{gather}
is of the form \eqref{FulSys} with 
\begin{gather*}
{\bf a}({\bf z},t)\equiv {\bf a}_0({\bf z})+t^{-\frac{h}{q}}{\bf a}_{h}({\bf z})+t^{-\frac{h+d}{q}}{\bf a}_{h+d}({\bf z}), \quad 
{\bf A}({\bf z},t)\equiv t^{-\frac{p}{q}} {\bf A}_p({\bf z}), \\
{\bf a}_0({\bf z})\equiv \begin{pmatrix} y \\ -x\end{pmatrix}, \quad 
{\bf a}_h({\bf z})\equiv \begin{pmatrix} 0 \\ F_1(x,y) \end{pmatrix}, \quad
{\bf a}_{h+d}({\bf z})\equiv \begin{pmatrix} 0 \\ F_{2}(x,y) \end{pmatrix}, \quad 
{\bf A}_p({\bf z})\equiv \begin{pmatrix} 0 & 0 \\ 0 & c\end{pmatrix}.
\end{gather*}
It follows that the corresponding limiting system \eqref{LimSys} with $H_0(x,y)=(x^2+y^2)/2$ has a stable equilibrium at the origin and $2\pi$-periodic solutions wtih $\nu(E)\equiv 1$.

{\bf 1.} Let $h=d=1$ and $p=q=2$. The changes of the variables described in Section~\ref{Sec3} with $N=4$ and
\begin{align*}
v_1(E,\varphi)   = &   -\left\{\int\limits_0^\varphi \left\{Y(\varsigma,E) F_1(X(\varsigma,E),Y(\varsigma,E)) \right\}_\varsigma\,d\varsigma\right\}_\varphi,\\
v_2(E,\varphi)   = &    -\left\{\int\limits_0^\varphi \left\{Y(\varsigma,E) F_2( Y(\varsigma,E)) + R_2(E,\varsigma)\right\}_\varsigma\,d\varsigma\right\}_\varphi, \\ 
v_3(E,\varphi) = & -\left\{\int\limits_0^\varphi \left\{R_3(E,\varsigma)\right\}_\varsigma\,d\varsigma\right\}_\varphi,\quad
v_4(E,\varphi) = -\left\{\int\limits_0^\varphi \left\{ R_4(E,\varsigma)\right\}_\varsigma\,d\varsigma\right\}_\varphi
\end{align*}
transform system \eqref{Ex2} into the form \eqref{Veq} with  
\begin{align*}
\Lambda_1(v)\equiv & \left\langle Y(\varphi,v) F_1(X(\varphi,v),Y(\varphi,v))\right\rangle_\varphi=\frac{a v^2}{2} (1+\mathcal O(v)),\\
\Lambda_2(v)\equiv & \left\langle Y(\varphi,v) F_2(X(\varphi,v),Y(\varphi,v))+ R_2(v,\varphi)\right\rangle_\varphi=b v(1+\mathcal O(v)),\\
\Lambda_3(v)\equiv & \left\langle R_3(v,\varphi)\right\rangle_\varphi=\mathcal O(v),\\
\Lambda_4(v)\equiv & \left\langle \frac{c^2}{2}+R_4(v,\varphi)\right\rangle_\varphi=\frac{c^2}{2}+\mathcal O(v)
\end{align*}
as $v\to 0$, where
\begin{align*}
R_2  \equiv &  - v_1\partial_E\Lambda_1+F_1(X ,Y ) \left( Y\partial_E-\frac{X}{2E}\partial_\varphi \right) v_1,\\
R_3  \equiv &  - v_1\left(\frac{1}{2}+\partial_E\Lambda_2\right) -v_2\partial_E\Lambda_1  -\frac{v_1^2}{2}\partial^2_E\Lambda_1  +F_1(X ,Y ) \left( Y\partial_E-\frac{X}{2E}\partial_\varphi \right) v_2\\
&+F_2(Y ) \left( Y\partial_E-\frac{X}{2E}\partial_\varphi \right) v_1,\\
R_4 \equiv &  -v_1\partial_E\Lambda_3-v_2(1+\partial_E\Lambda_2)-v_3\partial_E\Lambda_1 -\frac{v_1^2}{2}\partial^2_E\Lambda_2-\frac{v_1^3}{6}\partial^3_E\Lambda_1+F_1(X,Y)\left( Y\partial_E-\frac{X}{2E}\partial_\varphi \right) v_3 \\
& +F_2(Y)\left( Y\partial_E-\frac{X}{2E}\partial_\varphi \right) v_2.
\end{align*}
Hence, the transformed system satisfies \eqref{ass1} and \eqref{ass22} with $n=1$, $m=2$, $d=1$, $\lambda_{n,m}=a/2$, $\lambda_{n+d}=b$, $\mu_{2p}=c^2/2$, $(n+d,p,q)\in \Sigma_1$. Moreover, $(2p-n)/(2p-n-d)=3/2<m$, $\vartheta_1=1/2$, $\xi_1=|2b+1|/|a|$. It follows from Theorem~\ref{Th3} ({\bf Case I}) that if $a<0$ and $b+1/2>0$, then $H_0(x(t),y(t))\approx t^{-1/2}\xi_1$ as $t\to\infty$ with high probability for solutions of system \eqref{Ex2} with initial data such that $|t_0^{1/2} H_0(x(t_0),y(t_0))-\xi_1|$ is sufficiently small (see Fig.~\ref{FigEx21}). 

\begin{figure}
\centering
 {\includegraphics[width=0.4\linewidth]{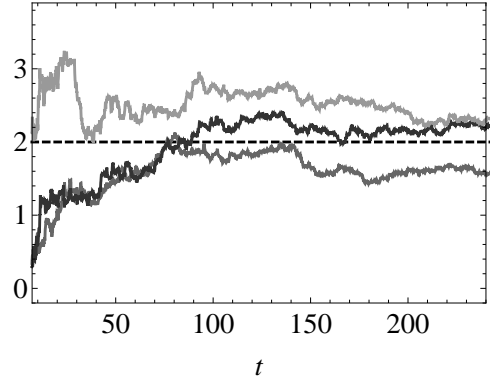}}
\caption{\small The evolution of $t^{1/2} H_0(x(t),y(t))$ for sample paths of the solutions to system \eqref{Ex2} with $h=d=1$, $p=q=2$, $a=-1$, $b=0.5$, $c=1$. The dashed line corresponds to $\xi_1=2$.} \label{FigEx21}
\end{figure}

{\bf 2.} Let $h=1$, $d=p=2$ and $q=3$. It can easily be checked that the changes of the variables described in Section~\ref{Sec3} with $N=4$ and
\begin{align*}
v_1(E,\varphi)  = &   -\left\{\int\limits_0^\varphi \left\{Y(\varsigma,E) F_1(X(\varsigma,E),Y(\varsigma,E)) \right\}_\varsigma\,d\varsigma\right\}_\varphi,\quad
v_2(E,\varphi)  =   -\left\{\int\limits_0^\varphi \left\{R_2(E,\varsigma)\right\}_\varsigma\,d\varsigma\right\}_\varphi, \\
v_3(E,\varphi)  = &  -\left\{\int\limits_0^\varphi \left\{Y(\varsigma,E) F_2( Y(\varsigma,E)) + R_3(E,\varsigma)\right\}_\varsigma\,d\varsigma\right\}_\varphi, \quad
v_4(E,\varphi)  =   -\left\{\int\limits_0^\varphi \left\{ R_4(E,\varsigma)\right\}_\varsigma\,d\varsigma\right\}_\varphi
\end{align*}
transform \eqref{Ex2} into \eqref{Veq} with  
\begin{align*}
\Lambda_1(v) =&\left\langle Y(\varphi,v) F_1(X(\varphi,v),Y(\varphi,v))\right\rangle_\varphi=\frac{a v^2}{2} (1+\mathcal O(v)),\\ 
\Lambda_2(v) =& \left\langle  R_2(v,\varphi)\right\rangle_\varphi=\mathcal O(v^3),\\
\Lambda_3(v)  =&\left\langle Y(\varphi,v) F_2(X(\varphi,v),Y(\varphi,v))+R_3(v,\varphi)\right\rangle_\varphi=bv (1+\mathcal O(v)),\\ 
\Lambda_4(v) =& \left\langle\frac{c^2}{2}+ R_4(v,\varphi)\right\rangle_\varphi=\frac{c^2}{2}+\mathcal O(v)
\end{align*}
as $v\to 0$,
where
\begin{align*}
R_2  \equiv &  - v_1\partial_E\Lambda_1+F_1(X ,Y ) \left( Y\partial_E-\frac{X}{2E}\partial_\varphi \right) v_1,\\
R_3  \equiv &  - v_1\partial_E\Lambda_2 -v_2\partial_E\Lambda_1  -\frac{v_1^2}{2}\partial^2_E\Lambda_1 +F_1(X ,Y ) \left( Y\partial_E-\frac{X}{2E}\partial_\varphi \right) v_2,\\
R_4 \equiv &  -v_1\left(\frac{1}{3}+\partial_E\Lambda_3\right)-v_2\partial_E\Lambda_2-v_3\partial_E\Lambda_1 -\frac{v_1^2}{2}\partial^2_E\Lambda_2-\frac{v_1^3}{6}\partial^3_E\Lambda_1 + F_1(X,Y)\left( Y\partial_E-\frac{X}{2E}\partial_\varphi \right) v_3\\
  & +F_2(Y)\left( Y\partial_E-\frac{X}{2E}\partial_\varphi \right) v_1.
\end{align*}
It follows that the transformed system satisfies \eqref{ass1} and \eqref{ass22} with $n=1$, $m=2$, $d=2$, $\lambda_{n,m}=a/2$, $\lambda_{n+d}=b$, $\mu_{2p}=c^2/2$, $(n+d,p,q)\in \Sigma_1$. In this case,  $(2p-n)/(2p-n-d)=3>m$, $\vartheta_2=1/2$, $\xi_2=|c|/\sqrt{|a|}$. By applying Theorem~\ref{Th3} ({\bf Case II}), we see that if $a<0$, then $H_0(x(t),y(t))\approx t^{-1/2}\xi_2$ as $t\to\infty$ with high probability for solutions of system \eqref{Ex2} with initial data such that $|t_0^{1/2} H_0(x(t_0),y(t_0))-\xi_2|$ is small enough (see Fig.~\ref{FigEx22}). 

\begin{figure}
\centering
 {\includegraphics[width=0.4\linewidth]{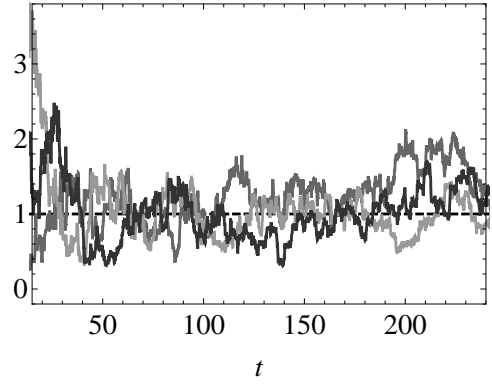}}
\caption{\small The evolution of $t^{1/2} H_0(x(t),y(t))$ for sample paths of the solutions to system \eqref{Ex2} with $h=1$, $d=p=2$, $q=3$, $a=-1$, $b=0.5$, $c=1$. The dashed line corresponds to $\xi_2=1$.} \label{FigEx22}
\end{figure}

{\bf 3.} Let $h=2$, $d=1$, $p=2$ and $q=3$. The changes of the variables described in Section~\ref{Sec3} with $N=4$, $v_1(E,\varphi)\equiv 0$, and
\begin{align*}
v_2(E,\varphi)  = &   -\left\{\int\limits_0^\varphi \left\{Y(\varsigma,E) F_1(X(\varsigma,E),Y(\varsigma,E)) \right\}_\varsigma\,d\varsigma\right\}_\varphi,\\
v_3(E,\varphi)  = &  -\left\{\int\limits_0^\varphi \left\{Y(\varsigma,E) F_2( Y(\varsigma,E)\right\}_\varsigma\,d\varsigma\right\}_\varphi, \quad
v_4(E,\varphi)  =    -\left\{\int\limits_0^\varphi \left\{ R_4(E,\varsigma)\right\}_\varsigma\,d\varsigma\right\}_\varphi 
\end{align*}
transform system \eqref{Ex2} into the form \eqref{Veq} with $\Lambda_1(v)\equiv 0$,
\begin{align*}
\Lambda_2(v) =&\left\langle Y(\varphi,v) F_1(X(\varphi,v),Y(\varphi,v))\right\rangle_\varphi=\frac{a v^2}{2} (1+\mathcal O(v)),\\ 
\Lambda_3(v) =& \left\langle  Y(\varphi,E) F_2( Y(\varphi,E)\right\rangle_\varphi=bv (1+\mathcal O(v)),\\
\Lambda_4(v) =& \left\langle\frac{c^2}{2}+ R_4(v,\varphi)\right\rangle_\varphi=\frac{c^2}{2}+\mathcal O(v)
\end{align*}
as $v\to 0$,
where
\begin{align*}
R_4(E,\varphi) \equiv &  F_1(X(\varphi,E),Y(\varphi,E))\left( Y(\varphi,E)\partial_E-\frac{X(\varphi,E)}{2E}\partial_\varphi \right) v_2(E,\varphi).
\end{align*}
Hence the transformed system satisfies \eqref{ass1} and \eqref{ass22} with $n=2$, $m=2$, $d=1$, $\lambda_{n,m}=a/2$, $\lambda_{n+d}=b$, $\mu_{2p}=c^2/2$, $(n+d,p,q)\in \Sigma_1$. It can easily be checked that  $(2p-n)/(2p-n-d)=2=m$ and $\vartheta_2=1/3$. Apllying Theorem~\ref{Th3} ({\bf Case III}) shows that if $a<0$ or $a>0$, $b+1/3<0$ and $c^2<(b+1/3)^2/a$, then there exists $\xi_3>0$ such that $H_0(x(t),y(t))\approx t^{-1/3}\xi_3$ as $t\to\infty$ with high probability for solutions of system \eqref{Ex2} with initial data such that $|t_0^{1/3} H_0(x(t_0),y(t_0))-\xi_3|$ is sufficiently small (see Fig.~\ref{FigEx23}). From Lemma~\ref{Lem3} it follows that 
\begin{gather*}
\xi_3=
\begin{cases}\displaystyle
\left|\frac{3b+1 }{3a}\right|\left({\hbox{\rm sgn}} (3b+1) +\sqrt{1+\frac{9|a|c^2}{(3b+1)^2}}\right) & \text{if } \quad a<0\\
\displaystyle
\frac{|3b+1| }{3a}\left(1 -\sqrt{1-\frac{9|a|c^2}{(3b+1)^2}}\right) & \text{if } \displaystyle \quad a>0, \quad b+\frac{1}{3}<0, \quad c^2<\frac{(3b+1)^2}{9a}
\end{cases}.
\end{gather*} 

\begin{figure}
\centering
\subfigure[$a=-1$, $b=2/3$, $c=1$]{\includegraphics[width=0.4\linewidth]{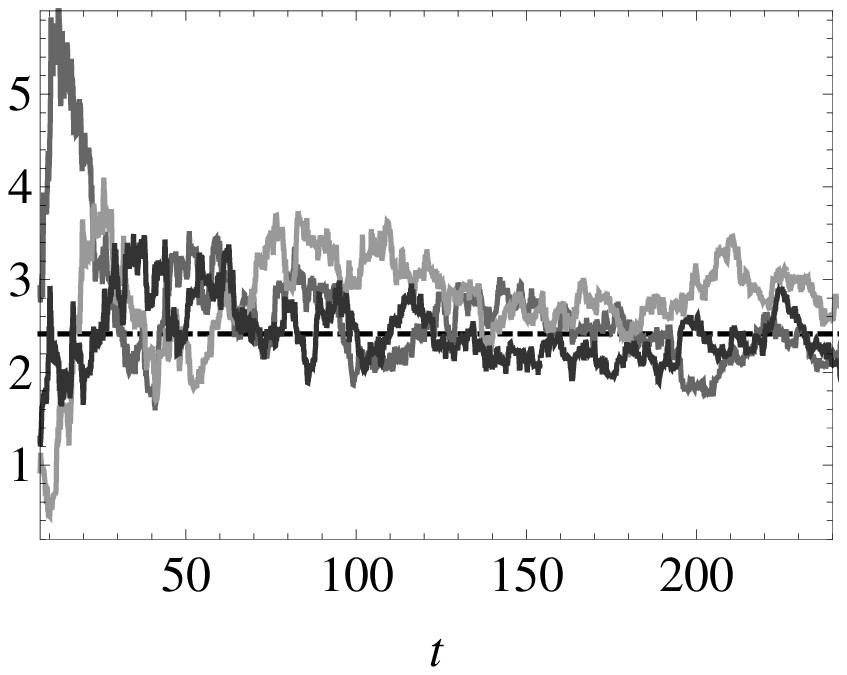}}
\hspace{4ex}
\subfigure[$a=1$, $b=-4/3$, $c=\sqrt{0.5}$]{\includegraphics[width=0.4\linewidth]{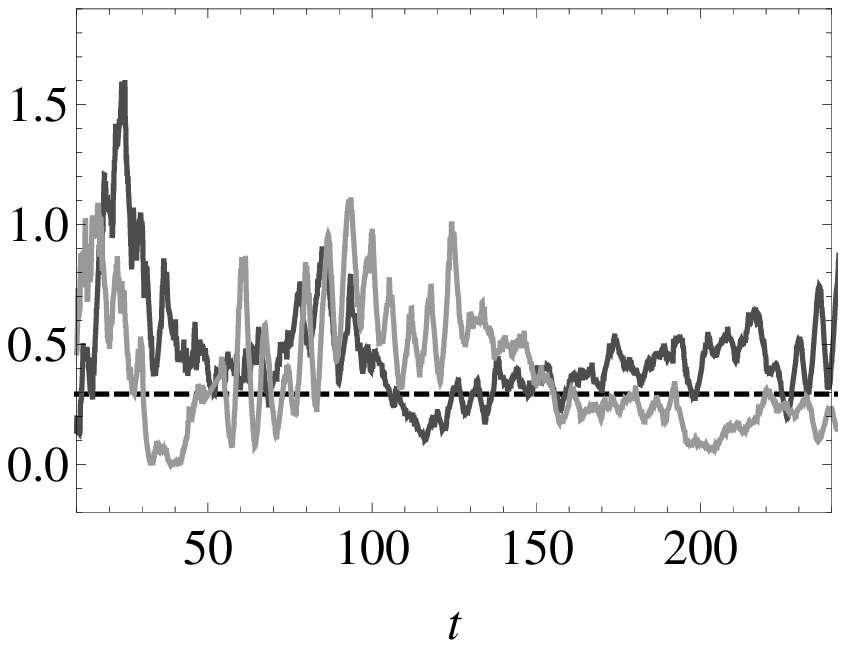}}
\caption{\small The evolution of $t^{1/3} H_0(x(t),y(t))$ for sample paths of the solutions to system \eqref{Ex2} with $h=2$, $d=1$, $p=2$, $q=3$. The dashed line corresponds to (a) $\xi_3=1+\sqrt 2$, (b) $\xi_3=1-\sqrt{0.5}$.} \label{FigEx23}
\end{figure}

\section{Application}\label{Appl}
In this section, the results obtained are applied to the model of parametric autoresonance with stochastic perturbations. Autoresonance is a phenomenon of a persistent phase-locking between an oscillatory nonlinear system and a small slowly-varying chirped-frequency excitation, resulting in a significant increase in the energy of the system~\cite{LFS09}. Such effects were first studied in problems related to acceleration of particles~\cite{VIV45,M45}, and over time it was found that autoresonance can be used in many applied problems~\cite{FGF00,LFSH16,GMetal17,LF19etal,LFAGS20}.

Consider the equations describing the initial stage of the autoresonant capture in systems with a parametric excitation~\cite{LKRMS08}:
\begin{gather}\label{PARsys}
\frac{d\Psi}{d\tau}=\mathcal E-a \tau + b \cos \Psi, \quad \frac{d \mathcal E}{d\tau}=\mathcal E(\sin \Psi - c),
\end{gather}
where $a,b,c={\hbox{\rm const}}>0$, $0<c<1$, $\Psi(\tau)$ and $\mathcal E(\tau)$ are unknown functions corresponding to the phase shift and energy of the oscillatory system, respectively. Solutions of system \eqref{PARsys} with $\Psi(\tau)\approx {\hbox{\rm const}}$ and $\mathcal E(\tau)\approx a \tau$ as $\tau\to\infty$ are associated with the phase-locking phenomenon and the capture into autoresonance. Note that system \eqref{PARsys} has a particular autoresonant solution $\Psi_\ast(\tau)$, $\mathcal E_\ast(\tau)$ with the following asymptotic expansion~\cite{OS19}:
\begin{gather}\label{parsol}
\Psi_\ast(\tau)=\sum_{j=0}^\infty \psi_j \tau^{-j}, \quad 
\mathcal E_\ast(\tau)=a\tau+\sum_{j=0}^\infty \mathcal E_j \tau^{-j}, \quad \tau\to\infty,
\end{gather}
where $\psi_j$ and $\mathcal E_j$ are constants. In particular, $\psi_0=\pi-\arcsin c$, $\psi_1=1/\cos\psi_0$, and $\mathcal E_0=-b \cos\psi_0$. Let us show that this solution is stable with respect to stochastic perturbations. Note that the stability of the parametric autoresonance with respect to small stochastic perturbations over a finite time interval was analysed in~\cite{OS19}. In this section, we discuss the long-term persistence of the capture. 

Consider the perturbed system in the form of It\^{o} stochastic differential equations
\begin{gather}\label{PARpert}
\begin{split}
& d\Psi = \left(\mathcal E-a \tau + b \cos \Psi\right) d\tau + \tilde{\alpha}_{1}\tau^{-\beta_1} d\tilde{w}_1(\tau),\\
& d\mathcal E = \mathcal E\left(\sin\Psi-c\right) d\tau + \tilde{\alpha}_{2}\tau^{-\beta_2} d\tilde{w}_2(\tau)
\end{split}
\end{gather}
as $\tau\geq \tau_0>0$, where $\tilde{\alpha}_1,\tilde{\alpha}_2,\beta_1,\beta_2$ are some constants and $(\tilde w_1(t),\tilde w_2(t))^T$ is a two-dimensional Wiener process on a probability space $(\Omega,\mathcal F,\mathbb P)$.  Substituting  
\begin{gather}\label{PARsubs}
\Psi(\tau)=\Psi_\ast(\tau)+ x(t), \quad \mathcal E(\tau)=\mathcal E_\ast(\tau)+ \gamma \tau^{\frac{1}{2}} y(t), \quad t=\frac{2\gamma}{3}\tau^{\frac{3}{2}}, \quad \gamma=\big(a^2(1-c^2)\big)^{\frac{1}{4}}>0
\end{gather}
into \eqref{PARpert} yields the asymptotically autonomous system
\begin{gather}\label{PARxy}
dx= P(x,y,t)dt+\alpha_1 t^{-\frac{1+4\beta_1}{6}} dw_1(t),\quad 
dy= Q(x,y,t)dt+\alpha_2 t^{-\frac{3+4\beta_2}{6}} dw_2(t),
\end{gather}
where 
\begin{gather*}
P(x,y,t)\equiv    y+ b t^{-\frac{1}{3}}\big(\cos(x+\Psi_\ast)-\cos\Psi_\ast\big)\left(\frac{2}{3\gamma^2}\right)^{\frac{1}{3}}, \\
Q(x,y,t)\equiv    \mathcal E_\ast t^{-\frac{2}{3}}\big(\sin(x+\Psi_\ast)-\sin\Psi_\ast\big)\left(\frac{2}{3\gamma^2}\right)^{\frac{2}{3}} + yt^{-\frac{1}{3}}\big(\sin(x+\Psi_\ast )-c\big)\left(\frac{2}{3\gamma^2}\right)^{\frac{1}{3}} - \frac{y}{3}t^{-1},\\
\alpha_1 =    \tilde \alpha_1  \left(\frac{2\gamma}{3}\right)^{\frac{1+4\beta_1}{6}}, \quad \alpha_2=\frac{\tilde \alpha_2}{\gamma}  \left(\frac{2\gamma}{3}\right)^{\frac{3+4\beta_2}{6}},
\end{gather*}
and $(w_1(t), w_2(t))^T$ is another two-dimensional Wiener process.  

{\bf 1.} Let $\beta_1=0$ and $\beta_2=-1/2$. Then system \eqref{PARxy} is of the form \eqref{FulSys} with $p=1$, $q=6$, 
\begin{gather*}
{\bf a}({\bf z},t)= {\bf a}_0({\bf z})+t^{-\frac{1}{3}}{\bf a}_2({\bf z})+\mathcal O(t^{-\frac{2}{3}}), \quad {\bf A}({\bf z},t)\equiv  t^{-\frac{1}{6}} {\bf A}_1({\bf z}), \quad  {\bf A}_1({\bf z})\equiv   \begin{pmatrix} \alpha_1 & 0 \\ 0 & \alpha_2\end{pmatrix},\\
{\bf a}_0({\bf z})\equiv  \begin{pmatrix} y \\ \frac{a}{\gamma^2}\big(\sin(x+\psi_0)-c\big)\end{pmatrix}, \quad 
{\bf a}_2({\bf z})\equiv \left(\frac{2}{3\gamma^2}\right)^{\frac{1}{3}} \begin{pmatrix} b \big(\cos (x+\psi_0)-\cos\psi_0\big) \\ y \big(\sin(x+\psi_0)-c\big) \end{pmatrix}.
\end{gather*}
The corresponding limiting system \eqref{LimSys} with 
\begin{gather*}
H_0(x,y)\equiv \frac{a}{\gamma^2}\Big(x\sin\psi_0+\cos(x+\psi_0)-\cos\psi_0\Big)+\frac{y^2}{2}=\frac{|{\bf z}|^2}{2}+\mathcal O(|{\bf z}|^3), \quad |{\bf z}|\to 0
\end{gather*}
 has a stable equilibrium at $(0,0)$, and the level lines $H_0(x,y)\equiv E$, lying in the neighbourhood of the equilibrium for a sufficiently small $E>0$, correspond to $T(E)$-periodic solutions such that 
\begin{gather*}
	\nu(E)\equiv\frac{2\pi}{T(E)}=1-\frac{12-7c^2}{96(1-c^2)} E + \mathcal O(E^2), \quad E\to 0.
\end{gather*}
The changes of the variables described in Section~\ref{Sec3} with $N=2$, $v_1(E,\varphi)\equiv 0$ and
\begin{align*}
v_2(E,\varphi) \equiv &   -\frac{1}{\nu(E)} \left(\frac{2}{3\gamma^2}\right)^{\frac{1}{3}} \left\{\int\limits_0^\varphi \left\{ \Big(Y^2 -  \frac{ab}{\gamma^2} \big[\cos(X+\psi_0)-\cos\psi_0\big]\Big)\big(\sin(X+\psi_0)-c\big)\right\}_\varsigma\,d\varsigma\right\}_\varphi\\
& +\frac{1}{2\nu(E)} \left\{\int\limits_0^\varphi \left\{  \frac{a\alpha_1^2}{\gamma^2} \cos(X+\psi_0)\right\}_\varsigma\,d\varsigma\right\}_\varphi
\end{align*}
transform system \eqref{PARxy} into the form \eqref{Veq} with $\Lambda_1(v)\equiv 0$ and 
\begin{align*}
\Lambda_2(v) =& \left(\frac{2}{3\gamma^2}\right)^{\frac{1}{3}} 
\left\langle 
\Big(\big(Y(\varphi,v)\big)^2 -  \frac{ab}{\gamma^2} \big[\cos(X(\varphi,v)+\psi_0)-\cos\psi_0\big]\Big)\big(\sin(X(\varphi,v)+\psi_0)-c\big)
\right \rangle_\varphi\\
& +  \frac{1}{2 }\left\langle 
 \alpha_2^2-\frac{a\alpha_1^2}{\gamma^2} \cos(X(\varphi,v)+\psi_0) 
\right \rangle_\varphi,
\end{align*}
where $X(\varphi,E)$, $Y(\varphi,E)$ is a $2\pi$-periodic solution of the corresponding limiting system \eqref{LimSys2pi}. It can easily be checked that
\begin{align*}
X(\varphi,E)=&\sqrt{2E}\cos\varphi+\frac{E a c }{6\gamma^2}\big(\cos (2 \varphi)-3\big) +\mathcal O(E^{\frac 32}), \\ 
Y(\varphi,E)=&-\sqrt{2E}\sin\varphi-\frac{E a c }{3\gamma^2}\sin (2 \varphi)  +\mathcal O(E^{\frac 3 2}) 
\end{align*}
as $E\to0$ uniformly for all $\varphi\in\mathbb R$. Hence,
\begin{gather*}
\Lambda_2(v)=  \frac{\alpha_1^2+\alpha_2^2}{2}-v\left[bc\left(\frac{2}{3\gamma^2}\right)^{\frac{1}{3}} + \frac{\alpha_1^2}{4(1-c^2)}\right]+\mathcal O(v^2), \quad v\to 0.
\end{gather*}
We see that the transformed system satisfies \eqref{ass1}  with $n=2$ and $(n,p,q)\in \Sigma_2$.
Moreover, for sufficiently small $\mu=(\alpha_1^2+\alpha_2^2)/2>0$ there exists $\xi_\ast>0$ such that $\Lambda_2(\xi_\ast)=0$ and $\Lambda_2'(\xi_\ast)<0$. In this case, assumption \eqref{ass3} holds with 
\begin{gather*}
\xi_\ast= \xi_\ast^{0}(1+\mathcal O(\mu)), \quad \Lambda_2'(\xi_\ast)=- bc\left(\frac{2}{3\gamma^2}\right)^{\frac{1}{3}}+\mathcal O(\mu), \quad \mu\to 0, \quad \xi_\ast^{0}=\frac{\mu}{bc}\left(\frac{3\gamma^2}{2}\right)^{\frac{1}{3}}.
\end{gather*}
By applying Theorem~\ref{Th4}, we find that $H_0(x(t),y(t))\approx \xi_\ast$ on an asymptotically long time interval for solutions of system \eqref{PARxy} with initial data such that $|H_0(x(t_0),y(t_0))-\xi_\ast|$ is sufficiently small.
Combining this with \eqref{parsol} and \eqref{PARsubs}, we obtain
\begin{gather*}
\Delta(\tau)\equiv |\Psi(\tau)-\pi+\arcsin c|+\tau^{-\frac{1}{2}}|\mathcal E(\tau)-a\tau|=\mathcal O(\tilde \mu), \quad \tilde \mu=\frac{\tilde\alpha_1^2+\tilde\alpha_2^2}{2}\to 0
\end{gather*}
with high probability on an asymptotically long time interval for solutions of system \eqref{PARpert} starting in the vicinity of $\Psi_\ast(\tau)$, $\mathcal E_\ast(\tau)$ (see Fig.~\ref{FigPAR1}).

\begin{figure}
\centering
\subfigure[]{
\includegraphics[width=0.3\linewidth]{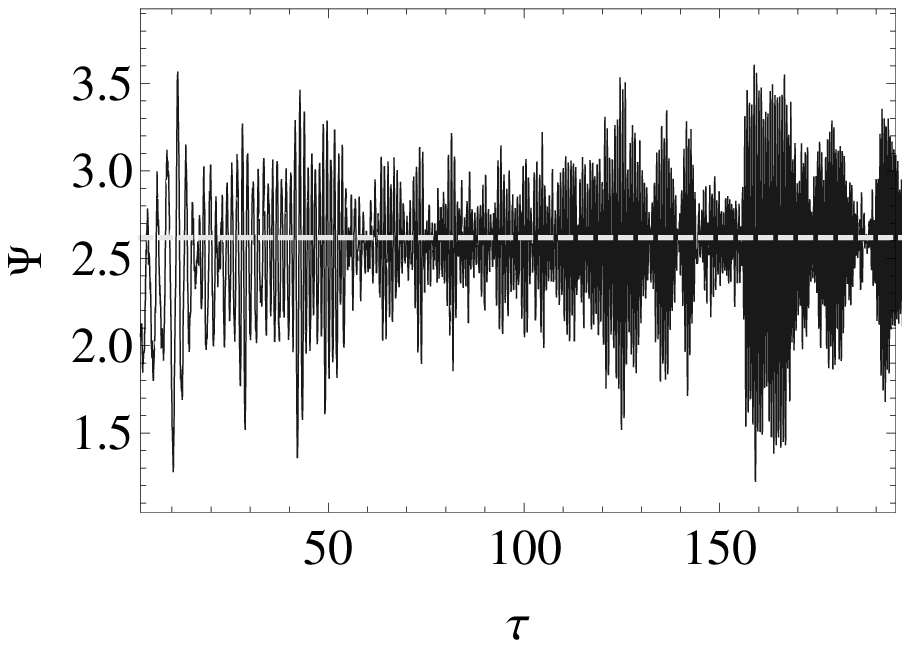}
}
\hspace{1ex}
 \subfigure[]{
\includegraphics[width=0.3\linewidth]{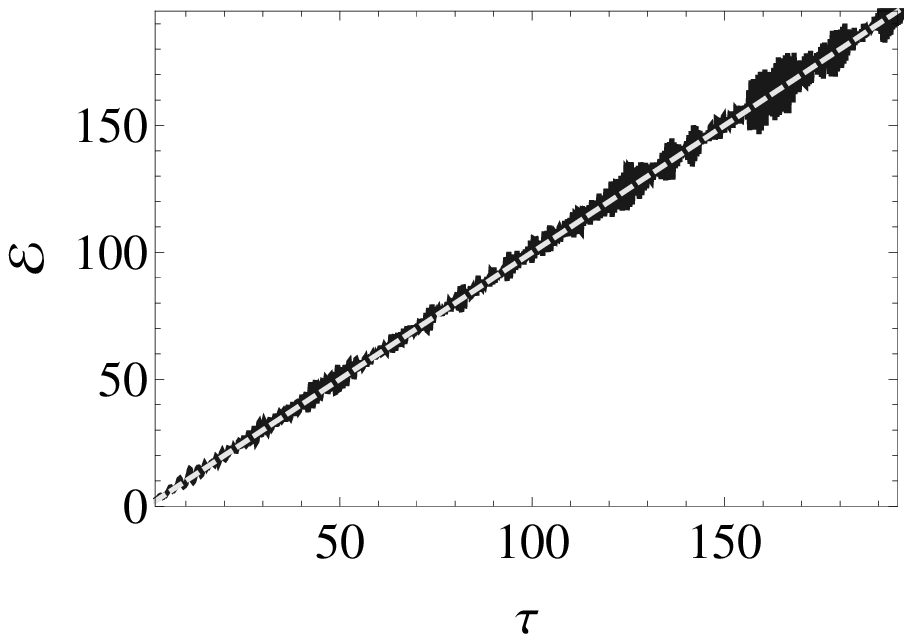}
}
\hspace{1ex}
\subfigure[]{
\includegraphics[width=0.3\linewidth]{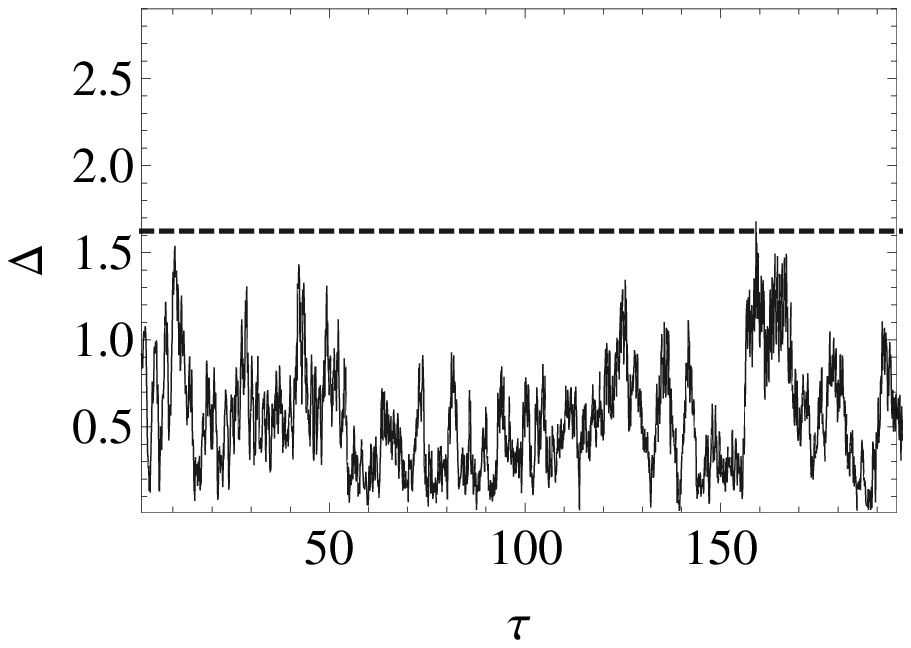}
}
\caption{\small The evolution of $\Psi(\tau)$, $\mathcal E(\tau)$ and $\Delta(\tau)$ for a sample path of the solution to system \eqref{PARpert} with $a=1$, $b=c=0.5$, $\tilde\alpha_1=0.25$, $\tilde\alpha_2=0.15$, $\beta_1=0$, $\beta_2=-1/2$. The dashed lines correspond to (a) $\psi_0=5\pi/6$, (b) $\tau$, (c) $4\sqrt{\xi^0_\ast}\approx 1.62$.} \label{FigPAR1}
\end{figure}

{\bf 2.} Now let $\beta_1=1/4$ and $\beta_2=-1/4$. In this case, system \eqref{PARxy} is of the form \eqref{FulSys} with $p=1$, $q=3$, 
\begin{gather*}
{\bf a}({\bf z},t)=   {\bf a}_0({\bf z})+t^{-\frac{1}{3}}{\bf a}_1({\bf z})+t^{-\frac{2}{3}}{\bf a}_2({\bf z})+\mathcal O(t^{-1}), \quad {\bf A}({\bf z},t)\equiv  t^{-\frac{1}{3}} {\bf A}_1({\bf z}), \quad 
{\bf A}_1({\bf z})\equiv   \begin{pmatrix} \alpha_1 & 0 \\ 0 & \alpha_2\end{pmatrix}, \\
{\bf a}_0({\bf z})\equiv  \begin{pmatrix} y \\ \frac{a}{\gamma^2}\big(\sin(x+\psi_0)-c\big)\end{pmatrix}, \quad 
{\bf a}_1({\bf z})\equiv \left(\frac{2}{3\gamma^2}\right)^{\frac{1}{3}} \begin{pmatrix} b \big(\cos (x+\psi_0)-\cos\psi_0\big) \\ y \big(\sin(x+\psi_0)-c\big) \end{pmatrix}, \\  
{\bf a}_2({\bf z})\equiv  \left(\frac{2}{3\gamma^2}\right)^{\frac{2}{3}} \begin{pmatrix}  0 \\ \mathcal E_0 \big(\sin(x+\psi_0)-c\big)+a\psi_1 \big(\cos(x+\psi_0)-\cos\psi_0\big) \end{pmatrix}.
\end{gather*}
Note that the transformations described in Section~\ref{Sec3} with $N=2$,
\begin{align*}
v_1(E,\varphi)  = &   -\frac{1}{\nu(E)} \left(\frac{2}{3\gamma^2}\right)^{\frac{1}{3}} \left\{\int\limits_0^\varphi \left\{ \Big(Y^2 -  \frac{ab}{\gamma^2} \big[\cos(X+\psi_0)-\cos\psi_0\big]\Big)\big(\sin(X+\psi_0)-c\big)\right\}_\varsigma\,d\varsigma\right\}_\varphi,\\
v_2(E,\varphi) = & -\frac{1}{\nu(E)} \left(\frac{2}{3\gamma^2}\right)^{\frac{2}{3}} \left\{\int\limits_0^\varphi \left\{   \mathcal E_0 Y\big(\sin(X+\psi_0)-c\big)+  a \psi_1 Y\big(\cos(X+\psi_0)-\cos\psi_0\big)\right\}_\varsigma\,d\varsigma\right\}_\varphi\\
& +\frac{1}{ \nu(E)} \left\{\int\limits_0^\varphi \left\{  \frac{a\alpha_1^2}{2\gamma^2} \cos(X+\psi_0)-R_2\right\}_\varsigma\,d\varsigma\right\}_\varphi
\end{align*}
reduce system \eqref{PARxy} to \eqref{Veq} with 
\begin{align*}
\Lambda_1(v) =& \left(\frac{2}{3\gamma^2}\right)^{\frac{1}{3}} 
\left\langle 
\Big(\big(Y(\varphi,v)\big)^2 -  \frac{ab}{\gamma^2} \big[\cos(X(\varphi,v)+\psi_0)-\cos\psi_0\big]\Big)\big(\sin(X(\varphi,v)+\psi_0)-c\big)
\right \rangle_\varphi,\\
\Lambda_2(v) =&    \left(\frac{2}{3\gamma^2}\right)^{\frac{2}{3}} 
\left\langle Y(\varphi,v)\Big (
 \mathcal E_0 \big(\sin(X(\varphi,v)+\psi_0)-c\big)+  a \psi_1 \big(\cos(X(\varphi,v)+\psi_0)-\cos\psi_0\big)
\Big)\right \rangle_\varphi\\
&  +  \left\langle 
 \frac{\alpha_2^2}{2}-\frac{a\alpha_1^2}{2\gamma^2} \cos(X(\varphi,v)+\psi_0)+R_2(v,\varphi)
\right \rangle_\varphi,
\end{align*}
and $R_2(E,\varphi) \equiv  (-\partial_E\Lambda_1+(\nabla_{\bf z} I)^T {\bf a}_1\partial_E +  (\nabla_{\bf z} \Phi)^T {\bf a}_1 \partial_\varphi )v_1(E,\varphi)$.
We see that 
\begin{gather*}
\Lambda_1(v)= v \left(-bc\left(\frac{2}{3\gamma^2}\right)^{\frac{1}{3}} +\mathcal O(v)\right),\quad
\Lambda_2(v)=\frac{\alpha_1^2+\alpha_2^2}{2}+\mathcal O(v), \quad v\to 0.
\end{gather*} 
Hence, the transformed system satisfies \eqref{ass1} and \eqref{ass21} with $n=1$, $\lambda_{n}=-bc (2/(3\gamma^2))^{1/3}$, $\mu_{2p}=(\alpha_1^2+\alpha_2^2)/2$, and $(n,p,q)\in \Sigma_1$. It can easily be checked that $\vartheta_0=1/3$ and $\xi_0=\mu_{2p}/|\lambda_n|$. Applying Theorem~\ref{Th2} shows that $H_0(x(t),y(t))\approx t^{-1/3}\xi_0$ as $t\to\infty$ with high probability for solutions of system \eqref{PARxy} with initial data such that $|t_0^{1/3} H_0(x(t_0),y(t_0))-\xi_0|$ is small enough. Therefore, by taking into account \eqref{parsol} and \eqref{PARsubs}, we obtain
\begin{gather*}
\Delta(\tau)\equiv |\Psi(\tau)-\pi+\arcsin c|+\tau^{-\frac{1}{2}}|\mathcal E(\tau)-a\tau|=\mathcal O(\tau^{-\frac 1 4}), \quad \tau\to\infty
\end{gather*}
with high probability for solutions of the perturbed system \eqref{PARpert} starting close to the particular solution $\Psi_\ast(\tau)$, $\mathcal E_\ast(\tau)$ (see Fig.~\ref{FigPAR2}). 

\begin{figure}
\centering
\subfigure[]{
\includegraphics[width=0.3\linewidth]{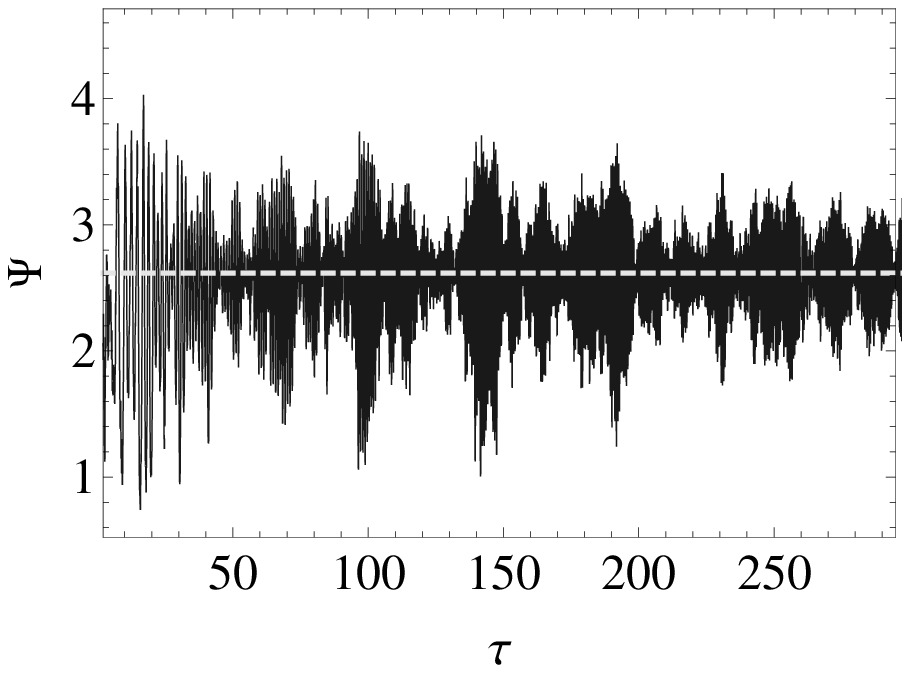}
}
\hspace{1ex}
 \subfigure[]{
\includegraphics[width=0.3\linewidth]{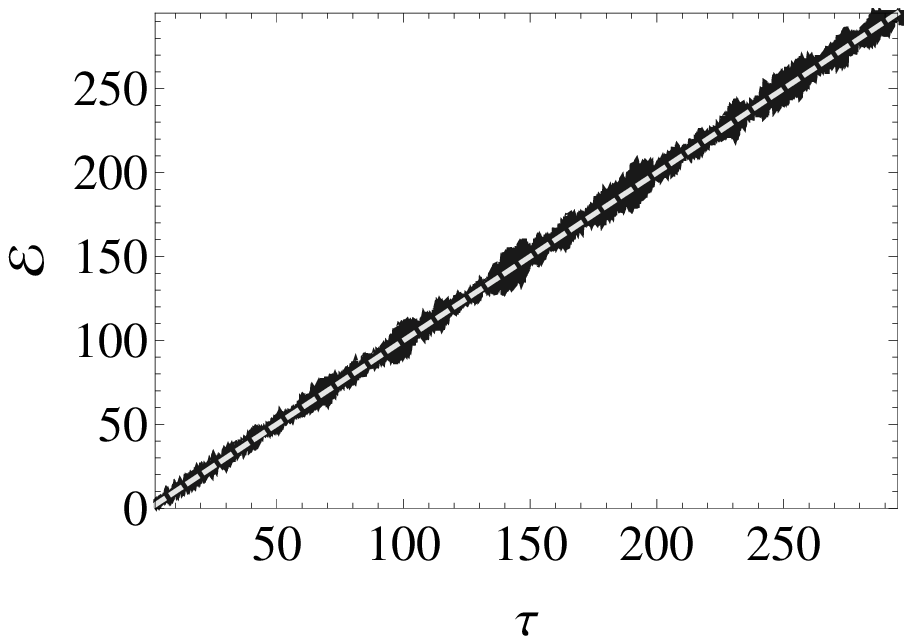}
}
\hspace{1ex}
\subfigure[]{
\includegraphics[width=0.3\linewidth]{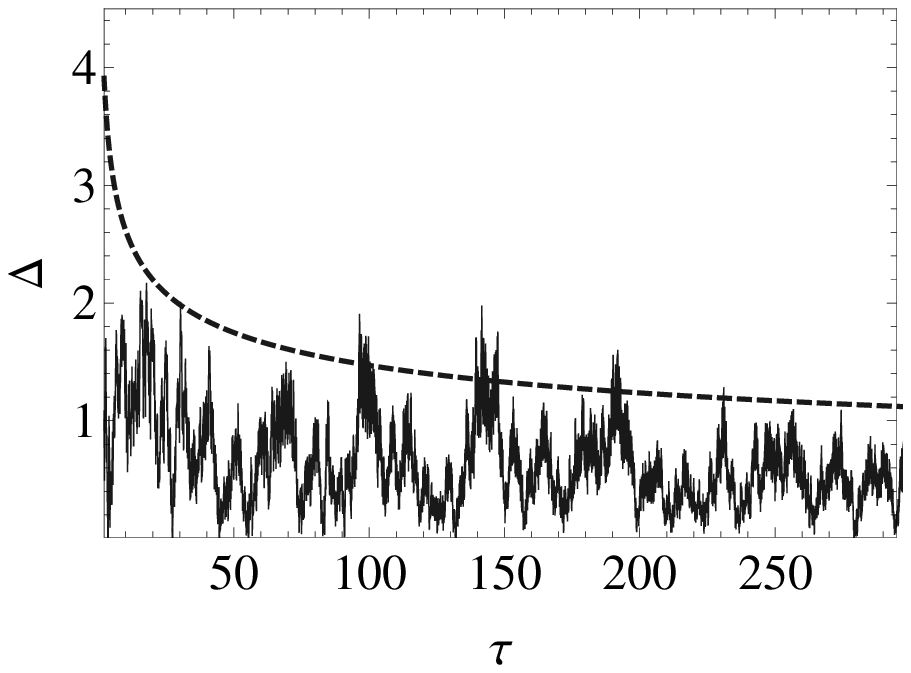}
}
\caption{\small The evolution of $\Psi(\tau)$, $\mathcal E(\tau)$ and $\Delta(\tau)$ for a sample path of the solution to system \eqref{PARpert} with $a=1$, $b=c=0.5$, $\tilde\alpha_1=0.75$, $\tilde\alpha_2=0.5$, $\beta_1=1/4$, $\beta_2=-1/4$. The dashed lines correspond to (a) $\psi_0=5\pi/6$, (b) $\tau$, (c) $4\sqrt{\xi_0} \tau^{-1/4}$, $\xi_0\approx 1.35$.} \label{FigPAR2}
\end{figure}

Thus, stochastic perturbations of the white noise type do not destroy the capture into parametric autoresonance.

\section{Conclusion}

Thus, possible stochastically stable asymptotic regimes in autonomous Hamiltonian systems with damped stochastic perturbations have been described. In particular, depending on the structure and parameters of decaying perturbations, the equilibrium of the limiting system can become asymptotically stable (Theorems~\ref{Th2} and \ref{Th3}) or new attractive states can appear (Theorem~\ref{Th4}). In these cases, bifurcations in the perturbed system \eqref{FulSys} are associated with the appearance of stable nontrivial solutions of the corresponding reduced equations for the normalized energy variable (see Lemmas~\ref{Lem1}, \ref{Lem3} and \ref{Lem4}). We have also described the conditions under which stochastically perturbed systems can behave like the corresponding limiting systems (Theorem~\ref{Th5}).  

The examples contained in sections~\ref{SecEx} and~\ref{Appl} illustrate the application of the results obtained.
Besides, the persistence of capture into parametric autoresonance under stochastic perturbations of white noise type has not been justified previously.

\section*{Acknowledgements}
The research is made in the framework of executing the development program of Volga Region Mathematical Center (agreement no. 075-02-2022-888).


}
\end{document}